\documentclass[12pt,a4paper]{wang}
\usepackage{etex}
\usepackage[numbers]{natbib}
\usepackage{amssymb,amstext,amscd,amsfonts,mathdots, amsmath}
\usepackage{hyperref}
\usepackage[vcentermath]{youngtab}
\usepackage{setspace}
\usepackage[all]{xy}
\usepackage{tikz}
\usetikzlibrary{shapes.geometric, arrows}
\usetikzlibrary{calc}
\usepackage{amsthm}
\usepackage{enumerate}
\usepackage{enumitem}
\setlist[description]{leftmargin=\parindent,labelindent=\parindent}
\usepackage{diagbox}
\usepackage{multirow,multicol}
\usepackage{pdflscape}
\usepackage{indentfirst}
\usepackage{ytableau}
\usepackage{color}
\usepackage{colortbl}
\definecolor{mygray}{gray}{0.80}
\definecolor{myblue}{rgb}{0.50,0.50, 0.95}
\definecolor{myred}{rgb}{0.99, 0.51, 0.65}
\usepackage{amsmath}
\usepackage{lipsum}
\let\OLDthebibliography\thebibliography
\renewcommand\thebibliography[1]{
  \OLDthebibliography{#1}
  \setlength{\parskip}{0pt}
  \setlength{\itemsep}{0.5pt}
}

\newtheorem{theorem}{Theorem}[section]
\newtheorem{theorem*}{Theorem}
\newtheorem{corollary}[theorem]{Corollary}
\newtheorem{corollary*}[theorem*]{Corollary}
\newtheorem{lemma}[theorem]{Lemma}
\newtheorem{proposition}[theorem]{Proposition}

\theoremstyle{definition}
\newtheorem{definition}[theorem]{Definition}
\newtheorem{remark}[theorem]{Remark}

\newtheorem*{question*}{Question}

\newtheorem*{conjecture*}{Conjecture}
\newtheorem{example}[theorem]{Example}

\newtheorem*{notation*}{Notation}

\newtheorem*{claim*}{Claim}
\newtheorem{definitiontheorem}[theorem]{Definition-Theorem}

\begin{document}
\begin{spacing}{1.2}
\title{\textbf{On $\tau$-tilting finite simply connected algebras}}
\author{Qi Wang\vspace{0.2cm} \\ \small{\ \emph{On the occasion of Professor Susumu Ariki's 60th birthday.}}}

\keywords{Support $\tau$-tilting modules, $\tau$-tilting finite, simply connected algebras.}
\abstract{\ \ \ \ We show that $\tau$-tilting finite simply connected algebras are representation-finite. Then, some related algebras are considered, including iterated tilted algebras, tubular algebras and so on. We also prove that the $\tau$-tilting finiteness of non-sincere algebras can be reduced to that of sincere algebras. This motivates us to give a complete
list of $\tau$-tilting finite sincere simply connected algebras.}

\maketitle
\section{Introduction}
In this paper, we always assume that $A$ is an associative finite-dimensional basic algebra with an identity over an algebraically closed field $K$. Moreover, the representation type of $A$ is divided into representation-finite, (infinite-)tame and wild.

In recent years, $\tau$-tilting theory introduced by Adachi, Iyama and Reiten \cite{AIR} has become increasingly important, where $\tau$ is the Auslander-Reiten translation for $A$. The core concept of $\tau$-tilting theory is the notion of support $\tau$-tilting modules: a right $A$-module $M$ is called support $\tau$-tilting if $\mathsf{Hom}_{A_M}(M,\tau M)=0$ and $\left | M \right |=\left | A_M \right |$ taking over $A_M:=A/A (1-e)A$, namely, $e$ is an idempotent of $A$ such that the direct summands of $eA/(e \mathsf{rad}\ A)$ are exactly the composition factors of $M$. Moreover, a support $\tau$-tilting $A$-module $M$ is called $\tau$-tilting if $A_M=A$, and any direct summand of $\tau$-tilting modules is called a $\tau$-rigid module. The class of support $\tau$-tilting modules is in bijection with the class of two-term silting complexes, left finite semibricks and so on. We refer to \cite{AIR} and \cite{Asai} for more details.

We are interested in $\tau$-tilting finite algebras studied in \cite{DIJ-tau-tilting-finite}, that is, algebras with finitely many pairwise non-isomorphic basic $\tau$-tilting modules (or equivalently, indecomposable $\tau$-rigid modules). It is easy to see that a representation-finite algebra is $\tau$-tilting finite. Also, it is not difficult to find a tame or a wild algebra which is $\tau$-tilting finite. The $\tau$-tilting finiteness for several classes of algebras has been determined, such as algebras with radical square zero \cite{Ada-rad-square-0}, preprojective algebras of Dynkin type \cite{Mizuno} and so on. In particular, it has been proven in some cases that $\tau$-tilting finiteness coincides with representation-finiteness, including gentle algebras \cite{P-gentle}, cycle-finite algebras \cite{MS-cycle-finite}, tilted and cluster-tilted algebras \cite{Z-tilted} and so on.

\paragraph{1.1 Motivations.} When we study the $\tau$-tilting finiteness of two-point algebras (i.e., algebras with only two simple modules) in \cite{two-point}, we find that the main variables depend on loops and oriented cycles. Thus, we want to see what happens if an algebra $A$ does not have multiple arrows, loops and oriented cycles, but has many vertices. For such algebras, we notice that the class of simply connected algebras is suitable since it is a rather large class of algebras and has been well-studied. Then, we focus on the $\tau$-tilting finiteness of simply connected algebras in this paper. In particular, we consider the triangle and rectangle quivers in detail as special cases.
\begin{center}
$\xymatrix@C=0.6cm@R=0.6cm{&&&\circ\ar[dr]&&&\\
&&\circ\ar[ur]\ar[dr]\ar@{.}[rr]&&\circ\ar[dr]&&\\
&\circ\ar[ur]\ar[dr]\ar@{.}[rr]&&\circ\ar[ur]\ar[dr]\ar@{.}[rr]&&\circ\ar[dr]&\\
\circ\ar[ur]&&\circ\ar[ur]&&\circ\ar[ur]&&\circ\\
&\vdots&&\vdots&&\vdots&}$\ \ \ \ \ \ \ \ \
$\xymatrix@C=0.7cm@R=0.6cm{\circ\ar[r]\ar[d]\ar@{.}[dr]&\circ\ar[r]\ar[d]\ar@{.}[dr]&\circ\ar[r]\ar[d]&\cdots\ar[r]&\circ\ar[r]\ar[d]\ar@{.}[dr]&\circ\ar[d]\\
\circ\ar[r]\ar[d]&\circ\ar[r]\ar[d]&\circ\ar[r]\ar[d]&\cdots\ar[r]&\circ\ar[r]\ar[d]&\circ\ar[d]\\
\vdots\ar[d]&\vdots\ar[d]&\vdots\ar[d]&\vdots&\vdots\ar[d]&\vdots\ar[d]\\
\circ\ar[r]\ar[d]\ar@{.}[dr]&\circ\ar[r]\ar[d]\ar@{.}[dr]&\circ\ar[r]\ar[d]&\cdots\ar[r]&\circ\ar[r]\ar[d]\ar@{.}[dr]&\circ\ar[d]\\
\circ\ar[r]&\circ\ar[r]&\circ\ar[r]&\cdots\ar[r]&\circ\ar[r]&\circ}$
\end{center}
\begin{center}
\textbf{Triangle quiver} \ \ \ \ \ \ \ \ \ \ \ \ \ \ \ \ \ \ \ \ \ \ \ \ \ \ \ \   \textbf{Rectangle quiver}
\end{center}
In these two special cases, the algebra $A\simeq KQ/I$ is presented by a triangle or a rectangle quiver $Q$ and an ideal $I$ generated by all possible commutativity relations. Here, a commutativity relation stands for the equality $w_1=w_2$ of two different paths $w_1$ and $w_2$ having the same source and target.

\paragraph{1.2 Main results.} The notion of simply connected algebras was first introduced by Bongartz and Gabriel \cite{BG} in representation-finite cases. The importance of these algebras is that we can reduce the representation theory of an arbitrary representation-finite algebra $A$ to that of a representation-finite simply connected algebra $B$. More precisely, for any representation-finite algebra $A$, the indecomposable $A$-modules can be lifted to indecomposable $B$-modules over a simply connected algebra $B$, which is contained inside a certain Galois covering of the standard form $\widetilde{A}$ of $A$, see \cite{BG-standard form} for details.

Soon after, Assem and Skowro$\acute{\text{n}}$ski \cite[Section 1.2]{AS-some-class} introduced the definition for an arbitrary algebra to be simply connected. In the case of representation-finite algebras, this new definition coincides with the definition in \cite{BG}. So we take this new definition in this paper (see Definition \ref{def-simply-connected}). Then, the class of simply connected algebras becomes rather large. For example, it includes tree algebras, tubular algebras, iterated tilted algebras of Euclidean type $\widetilde{\mathbb{D}}_n(n\geqslant 4)$, $\widetilde{\mathbb{E}}_6$, $\widetilde{\mathbb{E}}_7$, $\widetilde{\mathbb{E}}_8$ and so on.

The first main result of this paper is the following.
\begin{theorem}[Theorem \ref{result-strongly simply}]
Let $A$ be a simply connected algebra. Then, $A$ is $\tau$-tilting finite if and only if it is representation-finite.
\end{theorem}

We recall that an algebra $A$ is called sincere if there exists an indecomposable $A$-module $M$ such that all simple $A$-modules appear in $M$ as composition factors. Otherwise, $A$ is called non-sincere. Then, the second main result of this paper is given as follows.
\begin{theorem}[Theorem \ref{sincere-alge}]
Let $\left \{ e_1,e_2,\ldots,e_n \right \}$ be a complete set of pairwise orthogonal primitive idempotents of $A$. If $A$ is non-sincere, then $A$ is $\tau$-tilting finite if and only if $A/Ae_iA$ is $\tau$-tilting finite for any $1\leqslant i\leqslant n$.
\end{theorem}

Theorem 1.1 enables us to understand all $\tau$-tilting finite sincere simply connected algebras by quiver and relations, see Remark \ref{sincere-simply-connected}. Moreover, we point out that there are some inclusions as follows.
\begin{center}
$\begin{aligned}
&\{\text{Sincere simply connected algebras}\}\\
&\ \supseteq \{\text{Staircase algebras}\} \cup \{\text{Shifted-staircase algebras}\}\\
&\ \ \ \supseteq \begin{Bmatrix}
\text{Algebras presented by a triangle or a rectangle}\\
\text{quiver with all possible commutativity relations}
\end{Bmatrix}
\end{aligned}$
\end{center}
Then, the third main result of this paper indicates the $\tau$-tilting finiteness of the above subclasses, see Section 4 for details.

Lastly, as an example, we compute the number $\#\mathsf{s\tau\text{-}tilt}\ A$ of pairwise non-isomorphic basic support $\tau$-tilting $A$-modules for a specific $\tau$-tilting finite sincere simply connected algebra $A$. The motivation is explained in Section 5.

This paper is organized as follows. In Section 2, we review the definition of simply connected algebras as well as the definitions of critical algebras and Tits forms. In Section 3, we shall prove the first and second main results. This allows us to determine the $\tau$-tilting finiteness for several classes of algebras, such as iterated tilted algebras, tubular algebras and so on. In Section 4, we completely determine the $\tau$-tilting finiteness for algebras presented by a triangle or a rectangle quiver with all possible commutativity relations. In Section 5, we present an example to look at the number $\#\mathsf{s\tau\text{-}tilt}\ A$ for a specific $\tau$-tilting finite sincere simply connected algebra $A$.

\vspace{0.3cm}
\noindent\textit{Acknowledgements.} I am very grateful to my supervisor, Prof. Susumu Ariki, for giving me a lot of advice in writing this paper. Especially, many thanks to him for pointing out my logical gaps in the preparation of proofs. Also, I would like to thank Prof. Takuma Aihara, Kengo Miyamoto, Aaron Chan and Yingying Zhang for their many useful discussions on $\tau$-tilting theory. I am also very grateful to the anonymous referee for giving me so many useful suggestions and helping me improve this paper.

\section{Preliminaries}
Let $A\simeq KQ/I$ be an algebra with $Q=Q_A:=(Q_0,Q_1)$ over an algebraically closed field $K$. We may regard $A$ as a $K$-category (see \cite[Section 2]{BG}) which the class of objects is the set $Q_0$, and the class of morphisms from $i$ to $j$ is the $K$-vector space $KQ(i,j)$ of linear combinations of paths in $Q$ with source $i$ and target $j$, modulo the subspace $I(i,j):=I\cap KQ(i,j)$.

We recall some well-known definitions without further reference.
\begin{itemize}\setlength{\itemsep}{-3pt}
  \item $A$ is called triangular if $Q$ does not have loops and oriented cycles.
  \item A subcategory $B$ of $A$ is said to be full if for any $i,j\in Q_B$, every morphism $f:i\rightarrow j$ in $A$ is also in $B$; a full subcategory $B$ of $A$ is called convex if any path in $Q_A$ with source and target in $Q_B$ lies entirely in $Q_B$.
  \item A relation $\rho =\sum_{i=1}^{n}\lambda_i\omega_i \in I$ with $\lambda_i\neq 0$ and all $\omega_i$ have same source and target, is called minimal if $n\geqslant 2$ and $\sum_{j\in J}\lambda_j\omega_j \notin I$ for each non-empty proper subset $J\subset \left \{ 1,2,\dots,n \right \}$.
\end{itemize}

We introduce the definition of simply connected algebras as follows. Here, we follow the constructions in \cite{AS-some-class}. Let $A=KQ/I$ be a triangular algebra with a connected quiver $Q=(Q_0,Q_1, s, t)$ and an admissible ideal $I$. For each arrow $\alpha\in Q_1$, let $\alpha^-$ be its formal inverse with $s(\alpha^-)=t(\alpha)$ and $t(\alpha^-)=s(\alpha)$. Then, we set $Q_1^-:=\{\alpha^-\mid \alpha\in Q_1\}$.

A walk is a formal composition $w=w_1w_2\cdots w_n$ with $w_i\in Q_1\cup Q_1^-$ for all $1\leqslant i\leqslant n$. We set $s(w)=s(w_1)$, $t(w)=t(w_n)$ and denote by $1_x$ the trivial path at vertex $x$. For two walks $w$ and $u$ with $s(u)=t(w)$, the composition $wu$ is defined in the obvious way. In particular, $w=1_{s(w)}w=w1_{t(w)}$. Then, let $\sim$ be the smallest equivalence relation on the set of all walks in $Q$ satisfying the following conditions:
\begin{description}\setlength{\itemsep}{-3pt}
  \item[(1)] For each $\alpha:x\rightarrow y$ in $Q_1$, we have $\alpha\alpha^-\sim 1_x$ and $\alpha^-\alpha\sim 1_y$.
  \item[(2)] For each minimal relation $\sum_{i=1}^{n}\lambda_i\omega_i$ in $I$, we have $\omega_i\sim \omega_j$ for all $1\leqslant i,j\leqslant n$.
  \item[(3)] If $u,v,w$ and $w'$ are walks and $u\sim v$, we have $wuw'\sim wvw'$ whenever these compositions are defined.
\end{description}

We denote by $[w]$ the equivalence class of a walk $w$. Clearly, the product $wu$ of two walks $w$ and $u$ induces a product $[w]\cdot [u]$ of $[w]$ and $[u]$. Note that $[wu]=[w]\cdot [u]$.

For a given $x\in Q_0$, the set $\Pi _1(Q,I,x)$ of equivalence classes of all walks $w$ with $s(w)=t(w)=x$ becomes a group via the above product. Since $Q$ is connected, we can always find a walk $u$ from $x$ to $y$ for two different vertices $x$ and $y$, so that we can define an isomorphism from $\Pi _1(Q,I,x)$ to $\Pi _1(Q,I,y)$ by $[w]\longrightarrow [u]^{-1}\cdot[w]\cdot [u]$. This implies that $\Pi _1(Q,I,x)$ is independent of the choice of $x$, up to isomorphism. Then, the fundamental group of $(Q,I)$ is defined by $\Pi _1(Q,I):=\Pi _1(Q,I,x)$.

\begin{definition}(\cite[Definition 1.2]{AS-some-class}, \cite[(2.2)]{S-strongly-simply})\label{def-simply-connected}
Let $A$ be a triangular algebra.
\begin{description}\setlength{\itemsep}{-3pt}
  \item[(1)] $A$ is called simply connected if, for any presentation $A\simeq KQ/I$ as a bound quiver algebra, the fundamental group $\Pi_1(Q,I)$ is trivial.
  \item[(2)] $A$ is called strongly simply connected if every convex subcategory of $A$ is simply connected.
\end{description}
\end{definition}

We may distinguish the representation-finite cases as follows.
\begin{proposition}{\rm{(\cite[Corollary 2.8]{BG-standard form})}}\label{simply=stongly simply}
If $A$ is a representation-finite triangular algebra, $A$ is simply connected if and only if $A$ is strongly simply connected.
\end{proposition}

\begin{example}\label{example simple}
We have the following examples.
\begin{description}\setlength{\itemsep}{-3pt}
  \item[(1)] All tree algebras are strongly simply connected.
  \item[(2)] A hereditary algebra $KQ$ is (strongly) simply connected if and only if $Q$ is a tree.
  \item[(3)] Completely separating algebras are strongly simply connected, see \cite{Dr-complete-separating}.
  \item[(4)] Let $A:=KQ/I$ with $I:=\langle\alpha\beta-\gamma\delta, \alpha\lambda-\gamma\mu\rangle$ and the following quiver $Q$:
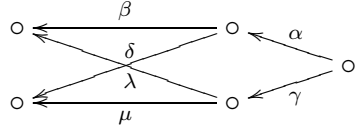
  \begin{center}
  $\xymatrix@C=1.1cm@R=0.1cm{\circ&&\circ\ar[ll]_{\beta}\ar[ddll]^{\lambda}&\\ &&&\circ\ar[ul]_{\alpha}\ar[dl]^{\gamma} \\ \circ &&\circ\ar[ll]^{\mu}\ar[lluu]_{\delta}&}$.
  \end{center}
  Then, $A$ is simply connected but not strongly simply connected, see \cite{A-introduction}.
\end{description}
\end{example}

We recall the separation property of a triangular algebra $A\simeq KQ/I$, which provides a sufficient condition for $A$ to be simply connected. Let $P_i$ be the indecomposable projective module at vertex $i$ and $\mathsf{rad}\ P_i$ its radical. Then, $P_i$ is said to have a separated radical (e.g., \cite[IX, Definition 4.1]{ASS}) if the supports of distinct indecomposable summands of $\mathsf{rad}\ P_i$ lie in distinct connected components of the full subquiver $Q(i)$ of $Q$ generated by the non-predecessors of $i$. We say that $A$ satisfies the separation property if every indecomposable projective $A$-module $P$ has a separated radical.

\begin{proposition}{\rm{(\cite[(2.3), (4.1)]{S-strongly-simply})}}\label{separation}
Let $A$ be a triangular algebra.
\begin{description}\setlength{\itemsep}{-3pt}
  \item[(1)] If $A$ satisfies the separation property, it is simply connected.
  \item[(2)] $A$ is strongly simply connected if and only if every convex subcategory of $A$ (or $A^\text{op}$) satisfies the separation property.
\end{description}
\end{proposition}

We recall that the one-point extension $A[M]$ of $A$ by an $A$-module $M$ is defined by
\begin{center}
$A[M]=\begin{bmatrix}
A &0\\
M&K
\end{bmatrix}$.
\end{center}
\begin{lemma}{\rm{(\cite[Lemma 2.3]{A-introduction})}}\label{one-point-extension}
If $A$ is simply connected and $M=\mathsf{rad}\ P$ is a separated radical of an indecomposable projective $A[M]$-module $P$, then $A[M]$ is simply connected.
\end{lemma}
\subsection{Critical algebras}\label{section-critical}
Let $\mathsf{mod}\ A$ be the category of finitely generated right $A$-modules and $\mathsf{proj}\ A$ the full subcategory of $\mathsf{mod}\ A$ consisting of projective $A$-modules. We denote by $\Gamma (\mathsf{mod}\ A)$ the Auslander-Reiten quiver of $A$. A connected component $C$ of $\Gamma (\mathsf{mod}\ A)$ is called preprojective if there is no oriented cycle in $C$, and any module in $C$ is of the form $\tau^{-n}(P)$ for an $n\in\mathbb{N}$ and an indecomposable projective $A$-module $P$, where $\tau$ is the Auslander-Reiten translation.

Let $|M|$ be the number of isomorphism classes of indecomposable direct summands of $M$. We recall from \cite{HR-tilted} that an $A$-module $T$ is called a tilting module if $|T|=|A|$, $\mathsf{Ext}_A^1(T,T)=0$ and the projective dimension $\mathsf{pd}_A\ T$ of $T$ is at most one. Dually, an $A$-module $T$ is called cotilting (\cite[Section 4.1]{R-tubular}) if it satisfies $\left | T \right |=\left | A \right |$, $\mathsf{Ext}_A^1(T,T)=0$ and the injective dimension $\mathsf{id}_A\ T$ is at most one. Let $A:=K\Delta$ be a hereditary algebra with a tilting $A$-module $T$, we call the endomorphism algebra $B:=\mathsf{End}_A\ T$ a tilted algebra of type $\Delta$. If moreover, $T$ is contained in a preprojective component $C$ of $\Gamma (\mathsf{mod}\ A)$, we call $B$ a concealed algebra of type $\Delta$.

\begin{definition}{\rm{(\cite[XX, Definition 2.8]{SS-text book})}}\label{minimal-infinite}
A critical algebra is one of concealed algebras of Euclidean type $\widetilde{\mathbb{D}}_n (n\geqslant 4)$, $\widetilde{\mathbb{E}}_6$, $\widetilde{\mathbb{E}}_7$ and $\widetilde{\mathbb{E}}_8$.
\end{definition}

\begin{remark}\label{critical-difference}
Note that the above definition is different from the original definition of critical algebras introduced by Bongartz in \cite{B-critical}: an algebra $A$ is called critical if $A$ is representation-infinite, but any proper convex subcategory of $A$ is representation-finite. In this paper, the critical algebras we use are actually the original critical algebras in \cite{B-critical} obtained by admissible gradings, as described in \cite[Theorem 2]{B-critical}.
\end{remark}

We mention that critical algebras have been classified completely by quiver and relations in both \cite{B-critical} and \cite{HV-tame-concealed}. We recall from \cite{HV-tame-concealed} that tame concealed algebras consist of critical algebras and the concealed algebras of Euclidean type $\widetilde{\mathbb{A}}_n$. Then, tame concealed algebras together with the generalized Kronecker algebras are precisely the minimal algebras of infinite representation type with a preprojective component. Here, an algebra $A$ is called a minimal algebra of infinite representation type if $A$ is representation-infinite, but $A/AeA$ is representation-finite for any non-zero idempotent $e$ of $A$.

We also point out that critical algebras are strongly simply connected. To show this, we first find in \cite[(1.2)]{AS-some-class} that critical algebras are simply connected. Since the quiver and relations of critical algebras are given in \cite{B-critical} (and \cite{HV-tame-concealed}), we observe that each proper convex subcategory of a critical algebra is also simply connected and hence, critical algebras are strongly simply connected (see also \cite[XX, Definition 2.8]{SS-text book}).

\subsection{Tits form}\label{tits-form}
Let $A=KQ/I$ be a triangular algebra and ${\bf{N}}:=\{0,1,2,\dots\}$. The Tits form (see \cite[Section 2]{Bo-tits-form}) $q_A: \mathbb{Z}^{Q_0}\rightarrow\mathbb{Z}$ of $A$ is the integral quadratic form defined by
\begin{center}
$q_A(v)=\sum\limits_{i\in Q_0}v_i^2-\sum\limits_{(i\rightarrow j) \in Q_1}v_iv_j+\sum\limits_{i,j\in Q_0}r(i,j)v_iv_j$,
\end{center}
where $v:=(v_i)\in \mathbb{Z}^{Q_0}$ and $r(i,j)=\left | R\cap I(i,j) \right |$ with a minimal set $R\subseteq \bigcup_{i,j\in Q_0}I(i,j)$ of generators of the admissible ideal $I$. Then, the Tits form $q_A$ is called weakly positive if $q_A(v)>0$ for any $v\neq 0$ in ${\bf{N}}^{Q_0}$, and weakly non-negative if $q_A(v)\geqslant 0$ for any $v\in {\bf{N}}^{Q_0}$.

It is well-known that the Tits form $q_A$ has a close connection with the representation type of $A$. Here, we recall the related results for (strongly) simply connected algebras.
\begin{proposition}{\rm{(\cite[XX, Theorem 2.9, 2.10]{SS-text book})}}\label{tits-form}
Let $A$ be a simply connected algebra.
\begin{description}\setlength{\itemsep}{-3pt}
  \item[(1)] $A$ is representation-finite if and only if the Tits form $q_A$ is weakly positive, or equivalently, if and only if $A$ does not contain a critical\footnote{As we mentioned in Remark \ref{critical-difference}, the definition of critical algebras used here comes from \cite[XX, Definition 2.8]{SS-text book}, but not from \cite{B-critical}.} convex subcategory.
  \item[(2)] If $A$ is strongly simply connected, $A$ is (infinite-)tame if and only if the Tits form $q_A$ is weakly non-negative, but not weakly positive.
\end{description}
\end{proposition}

\section{Simply connected algebras}
In this section, we first show that a $\tau$-tilting finite simply connected algebra is representation-finite. Then, we prove that the $\tau$-tilting finiteness of a non-sincere algebra can be reduced to that of a sincere algebra. Moreover, we get a complete list of $\tau$-tilting finite sincere simply connected algebras. Lastly, we determine the $\tau$-tilting finiteness of several algebras, which are related to the representation theory of (strongly) simply connected algebras. We need the following fundamental lemmas.
\begin{lemma}{\rm{(\cite[Theorem 5.12]{AIRRT}, \cite[Theorem 4.2]{DIJ-tau-tilting-finite})}}\label{quotient and idempotent}
If $A$ is $\tau$-tilting finite,
\begin{description}\setlength{\itemsep}{-3pt}
  \item[(1)] the quotient algebra $A/I$ is $\tau$-tilting finite for any two-sided ideal $I$ of $A$,
  \item[(2)] the idempotent truncation $eAe$ is $\tau$-tilting finite for any idempotent $e$ of $A$.
\end{description}
\end{lemma}

It is worth mentioning that if we consider $A$ as the $K$-category mentioned at the beginning of Section 2, then a convex subcategory $B$ of $A$ is actually a certain idempotent truncation of $A$. Thus, $B$ is $\tau$-tilting finite if $A$ is $\tau$-tilting finite.

\begin{lemma}{\rm{(\cite[Remark 2.9]{Mousavand-biserial alg})}}\label{crucial}
Let $A$ be an algebra with a preprojective component. Then, $A$ is $\tau$-tilting finite if and only if it is representation-finite.
\end{lemma}

\begin{lemma}\label{critical}
Any critical algebra $A$ is $\tau$-tilting infinite.
\end{lemma}
\begin{proof}
In our setting, a critical algebra is actually a representation-infinite concealed algebra. Such an algebra always admits a preprojective component in its Auslander-Reiten quiver, see \cite[VIII. Theorem 4.5]{ASS}. The statement follows from Lemma \ref{crucial}.
\end{proof}

Now, we are able to give the first main result of this section.
\begin{theorem}\label{result-strongly simply}
Let $A$ be a simply connected algebra. Then, $A$ is $\tau$-tilting finite if and only if $A$ is representation-finite.
\end{theorem}
\begin{proof}
If $A$ is representation-finite, it is obvious that $A$ is $\tau$-tilting finite. If $A$ is $\tau$-tilting finite, it cannot have a critical algebra as a convex subcategory by Lemma \ref{quotient and idempotent} and Lemma \ref{critical}. Then, $A$ is representation-finite following Proposition \ref{tits-form} (1).
\end{proof}

As we mentioned in Section 2, we have the following immediate results.
\begin{corollary}
Assume that $A$ is a simply connected algebra. If $A$ is not strongly simply connected, it is $\tau$-tilting infinite.
\end{corollary}
\begin{proof}
By Proposition \ref{simply=stongly simply}, such an algebra $A$ must be representation-infinite.
\end{proof}

\begin{corollary}
Let $A$ be a simply connected algebra and $B=A[M]$ the one-point extension with a separated radical $M=\mathsf{rad}\ P$ for an indecomposable projective $B$-module $P$. Then, $B$ is $\tau$-tilting finite if and only if $B$ is representation-finite.
\end{corollary}
\begin{proof}
This follows from the fact that $B$ is simply connected, see Lemma \ref{one-point-extension}.
\end{proof}

\begin{remark}
Let $A$ be a simply connected algebra. By Proposition \ref{tits-form} and Theorem \ref{result-strongly simply}, $A$ is $\tau$-tilting finite if and only if the Tits form $q_A$ is weakly positive, or equivalently, if and only if $A$ does not contain a critical convex subcategory.
\end{remark}

Next, we consider non-sincere and sincere algebras. Let $\left \{ e_1,e_2,\ldots,e_n \right \}$ be a complete set of pairwise orthogonal primitive idempotents of $A$. Then, we have
\begin{theorem}\label{sincere-alge}
A non-sincere algebra $A$ is $\tau$-tilting finite if and only if $A/Ae_iA$ is $\tau$-tilting finite for any $1\leqslant i\leqslant n$, if and only if each sincere quotient $A/AeA$ is $\tau$-tilting finite for an idempotent $e$ of $A$.
\end{theorem}
\begin{proof}
If $A$ is $\tau$-tilting finite, $A/Ae_iA$ is also $\tau$-tilting finite following Lemma \ref{quotient and idempotent} (1).

We assume that any $A/Ae_iA$ is $\tau$-tilting finite. Since $A$ is a non-sincere algebra, for any indecomposable $A$-module $M$, there exists at least one $e_i$ such that $Me_i=0$ and we may denote $B_i:=A/Ae_iA$. Then, for any indecomposable $\tau$-rigid $A$-module $M$, one can always find a suitable $i$ such that $M$ becomes an indecomposable $\tau$-rigid $B_i$-module. Besides, the number of indecomposable $\tau$-rigid $B_i$-modules is finite following the definition of $\tau$-tilting finite. Hence, $A$ is also $\tau$-tilting finite.
\end{proof}

According to the result above, the study for $\tau$-tilting finiteness of non-sincere algebras reduces to that of sincere algebras. We may apply this to simply connected algebras.

Let $A$ be a representation-finite sincere simply connected algebra. In \cite{Bo-sincere-simply}, Bongartz introduced a list of 24 infinite families containing all possible $A$'s with $\left | A \right |\geqslant 72$. Then, the bound is refined to $\left | A \right |\geqslant 14$ in Ringel's book \cite[Section 6]{R-tubular}. Hereafter, Bongartz determined the algebras with $\left | A \right |\leqslant 13$ by the graded tree introduced in \cite{BG}. Last, Rogat and Tesche \cite{RT-sincere-simply} gave a list of all possible $A$'s by their Gabriel quiver and relations.
\begin{remark}\label{sincere-simply-connected}
By Theorem \ref{result-strongly simply}, the list in \cite{RT-sincere-simply} provides a complete list of $\tau$-tilting finite sincere simply connected algebras.
\end{remark}

\subsection{Some applications}
We first consider the class of triangular matrix algebras. We denote by $\mathcal{T}_2(A)$ the algebra of $2\times 2$ upper triangular matrices $\bigl(\begin{smallmatrix}
A & A \\
0& A
\end{smallmatrix}\bigr)$ over an algebra $A$. Then, the category $\mathsf{mod}\ \mathcal{T}_2(A)$ is equivalent to the category whose objects are $A$-homomorphisms $f: M\rightarrow N$ between $A$-modules $M$ and $N$, and morphisms are pairs of homomorphisms making the obvious squares commutative. This reminds us that the category $\mathsf{mod}\ \mathcal{T}_2(A)$ is closely connected with the module category of the Auslander algebra of $A$.

Let $A$ be a representation-finite algebra and $\left \{ M_1,M_2,\dots,M_s \right \}$ a complete set of representatives of the isomorphism classes of indecomposable $A$-modules, the Auslander algebra of $A$ is defined as $\mathsf{End}_A\left ( \oplus_{i=1}^s M_i  \right )$. Then, we have the following result. During the preparation of this paper, we notice that the equivalence $(1)\Leftrightarrow (3)$ is directly proved in \cite[Theorem 3.1]{Aihara-Honma} by the $\tau$-rigid-brick correspondence.

\begin{proposition}\label{result-triangular-matrix}
Let $A$ be a representation-finite simply connected algebra and $B:=\mathsf{End}_A\left ( \oplus_{i=1}^s M_i  \right )$ its Auslander algebra, the following conditions are equivalent.
\begin{description}\setlength{\itemsep}{-3pt}
  \item[(1)] $B$ is $\tau$-tilting finite.
  \item[(2)] $B$ is representation-finite.
  \item[(3)] $\mathcal{T}_2(A)$ is $\tau$-tilting finite.
  \item[(4)] $\mathcal{T}_2(A)$ is representation-finite.
\end{description}
\end{proposition}
\begin{proof}
It follows from \cite[Theorem]{AB-auslander-algebra} that the Auslander algebra of $A$ is simply connected if and only if $A$ is simply connected. Therefore, $B$ is simply connected and $(1)\Leftrightarrow (2)$ follows from Theorem \ref{result-strongly simply}. It is known from \cite{LS-triangular-matrix} that $\mathcal{T}_2(A)\simeq \mathcal{T}_2(K)\otimes_K A$. Then, $\mathcal{T}_2(A)$ is simply connected if and only if $A$ is simply connected (see \cite{LS-tensor-product}). Hence, $(3)\Leftrightarrow (4)$ also follows from Theorem \ref{result-strongly simply}. Lastly, $(2)\Leftrightarrow (4)$ follows from \cite[Theorem 1.1]{AR} or \cite[VI, Proposition 5.8]{ARS}.
\end{proof}

We point out that if $A$ is a representation-infinite simply connected algebra, $\mathcal{T}_2(A)$ is $\tau$-tilting infinite. In fact, it is easy to check that $\mathcal{T}_2(A)$ has $A$ as a convex subcategory.

The second class of algebras we consider is the class of iterated tilted algebras. We recall that two algebras $A$ and $B$ are said to be tilting-cotilting equivalent if there exists a sequence of algebras $A=A_0,A_1,\dots,A_m=B$ and a sequence of modules $T_{A_i}^i$ with $0\leqslant i\leqslant m$ such that $A_{i+1}=\mathsf{End}_{A_i}\ T_{A_i}^i$ and $T_{A_i}^i$ is either a tilting or a cotilting module. Let $K\Delta$ be a hereditary algebra, $A$ is called iterated tilted of type $\Delta$ (see \cite[(1.4)]{AH-generalized-tilted} and \cite[Theorem 3]{HRS-iterated-tilted}) if $A$ is tilting-cotilting equivalent to $K\Delta$.

\begin{proposition}\label{result-iterated-tilted}
Let $A$ be an iterated tilted algebra of Dynkin type or types $\widetilde{\mathbb{D}}_n, \widetilde{\mathbb{E}}_p, (n\geqslant 4, p=6,7\ \text{or}\ 8)$. Then, $A$ is $\tau$-tilting finite if and only if it is representation-finite.
\end{proposition}
\begin{proof}
It is shown by \cite[Proposition 3.5]{A-iterated-tilted-B-D} that iterated tilted algebras of Dynkin type are simply connected, and by \cite[Corollary 1.4]{AS-some-class} that an iterated tilted algebra of Euclidean type is simply connected if and only if it is of types $\widetilde{\mathbb{D}}_n, \widetilde{\mathbb{E}}_p, (n\geqslant 4, p=6,7\ \text{or}\ 8)$. Then, the statement follows from Theorem \ref{result-strongly simply}.
\end{proof}

Lastly, we consider the algebras that have close connection with the representation type of strongly simply connected algebras. For the sake of simplicity, we have omitted their definitions and one can refer to the related articles for more details.
\begin{itemize}\setlength{\itemsep}{-3pt}
  \item Tubular algebras are introduced by Ringel in \cite[Chapter 5]{R-tubular}, which are simply connected (\cite[(1.4)]{AS-some-class}) and have only 6, 8, 9 or 10 simple modules. In particular, \cite[Proposition 2.4]{S-polynomial} implies that every tubular algebra is representation-infinite.
  \item The algebras which are derived equivalent to tubular algebras are also simply connected, see \cite{AS-some-class}. However, such an algebra could be representation-finite. One may refer to \cite[Theorem]{Ba-derived-tubular} for an explicit characterization of representation-finite algebras which are derived equivalent to tubular algebras.
  \item $pg$-critical algebras are introduced by N$\ddot{\text{o}}$renberg and Skowro$\acute{\text{n}}$ski in \cite{NS-pg-critical}, these algebras stand for the $p$olynomial $g$rowth critical algebras, that is, representation-infinite tame simply connected algebras which are not of polynomial growth, but every proper convex subcategory is.
  \item Hypercritical algebras are introduced by Unger \cite{U-hyper} (see also Lersch \cite{L-hyper} and Wittman \cite{W-hyper}), these algebras are preprojective tilts of minimal wild hereditary tree algebras of some types. One can understand hypercritical algebras by quivers and relations in \cite{U-hyper} and they are strongly simply connected. Actually, they are minimal wild strongly simply connected algebras.
\end{itemize}

We point out that tubular algebras and $pg$-critical algebras are not necessarily strongly simply connected while hypercritical algebras must be strongly simply connected. We are able to determine their $\tau$-tilting finiteness by Theorem \ref{result-strongly simply}.
\begin{proposition}
All tubular, $pg$-critical and hypercritical algebras are $\tau$-tilting infinite.
\end{proposition}

\begin{proposition}
Let $A$ be an algebra which is derived equivalent to a tubular algebra. Then, $A$ is $\tau$-tilting finite if and only if it is representation-finite.
\end{proposition}

\section{Algebras with rectangle or triangle quiver}
In this section, we first recall the constructions of staircase algebras $\mathcal{A}(\lambda)$ introduced by Boos \cite{B-staircase}, which are parameterized by partitions $\lambda$. As we will show below, an algebra presented by a rectangle quiver with all possible commutativity relations is actually a staircase algebra $\mathcal{A}(\lambda)$ with $\lambda=(m^n)$. Similarly, we introduce the shifted-staircase algebra $\mathcal{A}^s(\lambda^s)$ as a generalization of algebras presented by triangle quivers.

\subsection{Staircase algebras}
We recall that a partition $\lambda=(\lambda_1\geqslant \lambda_2\geqslant \cdots\geqslant \lambda_\ell)$ of $n$ is a non-increasing sequence of positive integers such that $\sum_{i=1}^{\ell}\lambda_i=n$. Usually, we merge the same entries of $\lambda$ by potencies, for example, $(3,3,2,1,1)=(3^2,2,1^2)$. We denote by $Y(\lambda)$ the Young diagram of $\lambda$, that is, a box-diagram of which the $i$-th row contains $\lambda_i$ boxes. For example,
\begin{center}
$Y(3^2,2)=\yng(3,3,2)$.
\end{center}
Starting with $(1,1)$ in the top-left corner, we assign each of the boxes in $Y(\lambda)$ a coordinate $(i,j)$ by increasing $i$ from top to bottom and $j$ from left to right.

\begin{definition}(\cite[Definition 3.1]{B-staircase})
Let $\lambda$ be a partition. We define $Q_\lambda$ and $I_\lambda$ by
\begin{itemize}\setlength{\itemsep}{-3pt}
  \item the vertices of $Q_\lambda$ are given by the boxes appearing in $Y(\lambda)$.
  \item the arrows of $Q_\lambda$ are given by all $(i,j)\rightarrow (i,j+1)$ and $(i,j)\rightarrow (i+1,j)$, whenever all these vertices are defined.
  \item $I_\lambda$ is a two-sided ideal generated by all possible commutativity relations for all squares appearing in $Q_\lambda$.
\end{itemize}
Then, the bound quiver algebra $\mathcal{A}(\lambda):=KQ_\lambda/I_\lambda$ is called a staircase algebra.
\end{definition}

Following the above example, let $\lambda=(3^2,2)$, the associated quiver $Q_\lambda$ is given by
\begin{center}
$\vcenter{\xymatrix@C=0.7cm@R=0.6cm{(1,1)\ar[r]^{\alpha_{1,1}}\ar[d]_{\beta_{1,1}}& (1,2)\ar[d]_{\beta_{1,2}}\ar[r]^{\alpha_{1,2}}& (1,3)\ar[d]_{\beta_{1,3}}\\
(2,1)\ar[r]^{\alpha_{2,1}}\ar[d]_{\beta_{2,1}} & (2,2)\ar[d]_{\beta_{2,2}}\ar[r]^{\alpha_{2,2}}& (2,3)\\
(3,1)\ar[r]^{\alpha_{3,1}} & (3,2)&}}$ $\simeq$ $\vcenter{\xymatrix@C=0.7cm@R=0.6cm{\circ\ar[r]\ar[d]& \circ\ar[d]\ar[r]& \circ\ar[d]\\
\circ\ar[r]\ar[d]& \circ\ar[d]\ar[r]& \circ\\
\circ\ar[r] & \circ&}}$,
\end{center}
and the corresponding staircase algebra $\mathcal{A}(\lambda)$ is defined by
\begin{center}
$\mathcal{A}(\lambda):=KQ_\lambda/\langle\alpha_{1,1}\beta_{1,2}-\beta_{1,1}\alpha_{2,1},\alpha_{1,2}\beta_{1,3}-\beta_{1,2}\alpha_{2,2},\alpha_{2,1}\beta_{2,2}-\beta_{2,1}\alpha_{3,1}\rangle$.
\end{center}

We denote by $\lambda^T$ the transposed partition of a partition $\lambda$, which is given by the columns of the Young diagram $Y(\lambda)$ from left to right. Then, $\mathcal{A}(\lambda)$ is isomorphic to $\mathcal{A}(\lambda^T)$. Moreover, $\mathcal{A}(\lambda)$ is a basic, connected, triangular, finite-dimensional $K$-algebra.

\begin{proposition}{\rm{(\cite[Proposition 3.7]{B-staircase})}}
Let $\lambda$ be a partition. Then, the staircase algebra $\mathcal{A}(\lambda)$ is strongly simply connected.
\end{proposition}

This implies that $\mathcal{A}(\lambda)$ is $\tau$-tilting finite if and only if it is representation-finite. Since the author of \cite{B-staircase} has given a complete classification of the representation type of $\mathcal{A}(\lambda)$, we can understand all $\tau$-tilting finite staircase algebras by quiver and relations.
\begin{theorem}{\rm{(Modified, \cite[Theorem 4.5]{B-staircase})}}\label{stair-classification}
A staircase algebra $\mathcal{A}(\lambda)$ with a partition $\lambda$ of $n$ is $\tau$-tilting finite if and only if one of the following holds:
  \begin{itemize}\setlength{\itemsep}{-3pt}
    \item $\lambda\in\{(n), (n-k, 1^k), (n-2, 2), (2^2,1^{n-4})\}$ for $k\leqslant n$.
    \item $n\leqslant 8$ and $\lambda\notin\{(4,3,1),(3^2,2),(3,2^2,1),(4,2,1^2)\}$.
  \end{itemize}
\end{theorem}

Let $\vec{A}_n$ be the path algebra of Dynkin type $\mathbb{A}_n$ associated with linear orientation. Then, we define
\begin{center}
$\mathcal{B}:=\left \{ B_{m,n}\mid B_{m,n}\ \text{is the tensor product}\ \vec{A}_m\otimes_K \vec{A}_{n} \right \}$.
\end{center}
Note that $B_{m,n}\simeq B_{n,m}$ is presented by a rectangle quiver with all possible commutativity relations and vice versa. We also note that $B_{m,n}$ can be regarded as a special staircase algebra. Hence, we can determine the $\tau$-tilting finiteness of $B_{m,n}$ as follows.
\begin{corollary}\label{result-rectangle}
Let $B_{m,n}\in \mathcal{B}$. Then, the algebra $B_{m,n}$ is $\tau$-tilting finite if and only if $(m,n)$ or $(n,m)\in\left \{ (1,k),(2,2),(2,3),(2,4)\mid k\in \mathbb{N} \right \}$.
\end{corollary}
\begin{proof}
It is obvious that $B_{m,n}$ is a staircase algebra $\mathcal{A}(\lambda)$ with $\lambda=(m^n)$ or $(n^m)$. Then, the statement follows from Theorem \ref{stair-classification}.
\end{proof}

\subsection{Shifted-staircase algebras}
Now, we consider the triangle quivers. We point out that the quiver of staircase algebras cannot be a triangle quiver because of the different orientations. One may look at the following case as an example.
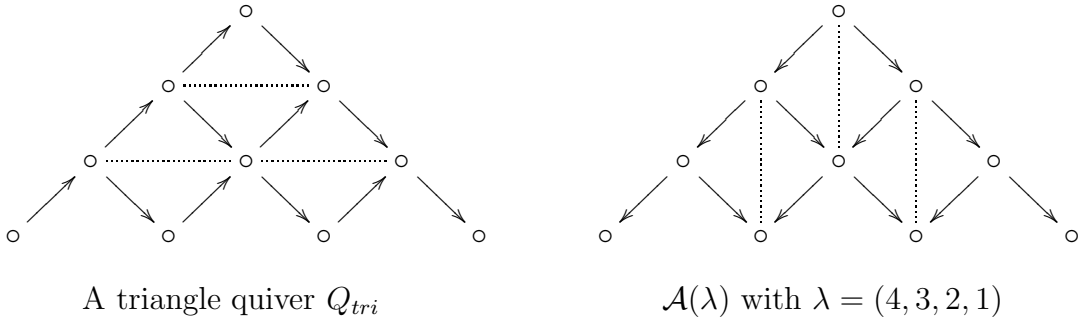
\begin{center}
\scalebox{0.8}{$\xymatrix@C=0.6cm@R=0.6cm{&&&\circ\ar[dr]&&&\\
&&\circ\ar[ur]\ar[dr]\ar@{.}[rr]&&\circ\ar[dr]&&\\
&\circ\ar[ur]\ar[dr]\ar@{.}[rr]&&\circ\ar[ur]\ar[dr]\ar@{.}[rr]&&\circ\ar[dr]&\\
\circ\ar[ur]&&\circ\ar[ur]&&\circ\ar[ur]&&\circ}$\ \ \ \ \ \ \ \ \
$\xymatrix@C=0.6cm@R=0.6cm{&&&\circ\ar[dr]\ar[dl]\ar@{.}[dd]&&&\\
&&\circ\ar[dl]\ar[dr]\ar@{.}[dd]&&\circ\ar[dr]\ar[dl]\ar@{.}[dd]&&\\
&\circ\ar[dl]\ar[dr] &&\circ\ar[dl]\ar[dr] &&\circ\ar[dr]\ar[dl]&\\
\circ &&\circ &&\circ &&\circ}$}
\end{center}
\begin{center}
\scalebox{0.8}{A triangle quiver $Q_{tri}$ \ \ \ \ \ \ \ \ \ \ \ \ \ \ \ \ \ \ \ \ \ \ \ \ \ \   $\mathcal{A}(\lambda)$ with $\lambda=(4,3,2,1)$}
\end{center}
This motivates us to introduce the shifted-staircase algebras.

We recall that a shifted partition $\lambda^s=(\lambda_1^s>\lambda_2^s>\cdots>\lambda_\ell^s)$ is a strictly decreasing sequence of positive integers. Similarly, we can visualize $\lambda^s$ by the shifted Young diagram $Y(\lambda^s)$, that is, a box-diagram of which the $i$-th row contains $\lambda_i$ boxes and is shifted to the right $i-1$ steps. For example, let $\lambda^s=(4,3,2,1)$, then
\begin{center}
$Y(\lambda^s)=$ \scalebox{0.6}{$\ydiagram{4,1+3,2+2,3+1}$}.
\end{center}
Starting with $(1,1)$ in the top-left corner, we assign each of the boxes in $Y(\lambda^s)$ a coordinate $(i,j)$ with $j\geqslant i$ by increasing $i$ from top to bottom and $j$ from left to right.

\begin{definition}\label{def-shifted-staircase}
Let $\lambda^s$ be a shifted partition and $\mathcal{A}^s(\lambda^s):=KQ_{\lambda^s}/I_{\lambda^s}$ such that
\begin{itemize}\setlength{\itemsep}{-3pt}
  \item the vertices of $Q_{\lambda^s}$ are given by the boxes appearing in $Y(\lambda^s)$.
  \item the arrows of $Q_{\lambda^s}$ are given by all $(i,j)\rightarrow (i,j+1)$ and $(i,j)\rightarrow (i+1,j)$, whenever all these vertices are defined.
  \item $I_{\lambda^s}$ is a two-sided ideal generated by all possible commutativity relations for all squares appearing in $Q_{\lambda^s}$.
\end{itemize}
Then, the bound quiver algebra $\mathcal{A}^s(\lambda^s)$ is called a shifted-staircase algebra.
\end{definition}

For example, the quiver $Q_{\lambda^s}$ with $\lambda^s=(4,3,2,1)$ is exactly the triangle quiver $Q_{tri}$ displayed above. In fact, the algebra presented by a triangle quiver with $\frac{n(n+1)}{2}$ vertices (and all possible commutativity relations), is exactly the shifted-staircase algebra $\mathcal{A}^s(\lambda^s)$ with $\lambda^s=(n, n-1, \ldots,2,1)$.

It is obvious that $\mathcal{A}^s(\lambda^s)$ is also a basic, connected, triangular, finite-dimensional $K$-algebra. Next, we show that $\mathcal{A}^s(\lambda^s)$ is strongly simply connected.
\begin{proposition}
For a shifted partition $\lambda^s$, $\mathcal{A}^s(\lambda^s)$ is strongly simply connected.
\end{proposition}
\begin{proof}
Let $B$ be a convex subcategory of $\mathcal{A}^s(\lambda^s)$ and $P_{(i,j)}$ the indecomposable projective $B$-module at vertex $(i,j)$. Then, one can check that $\mathsf{rad}\ P_{(i,j)}$ is either indecomposable or decomposed into exactly two indecomposable $B$-modules. The latter case appears if and only if $(i,j+1)$ and $(i+1,j)$ are vertices of the quiver $Q_B$ of $B$, but $(i+1,j+1)$ is not.

For the former case, there is nothing to prove. For the latter case, let $Q(i,j)$ be the subquiver of $Q_B$ obtained by deleting all vertices of $Q_B$ being a source of a path in $Q_B$ with target $(i,j)$. Then, $Q(i,j)$ is decomposed into two disjoint subquivers such that $\mathsf{rad}\ P_{(i,j)}$ is separated. Hence, $B$ satisfies the separation condition. Then, the statement follows from Proposition \ref{separation} (2).
\end{proof}

Now, we have understood that $\mathcal{A}^s(\lambda^s)$ is $\tau$-tilting finite if and only if it is representation-finite. Then, it is reasonable to give a complete classification for the representation type of shifted-staircase algebras. Before to do this, we need the following observation.

Let $\lambda^s=(\lambda_1^s>\lambda_2^s>\cdots>\lambda_\ell^s)$ and $\mu^s=(\mu_1^s>\mu_2^s>\cdots>\mu_k^s)$ be two shifted partitions. We say that $\lambda^s\leq \mu^s$ if $\ell \leqslant k$ and $\lambda_i^s\leqslant \mu_i^s$ for all $1\leqslant i\leqslant \ell$.
\begin{proposition}
Suppose that $\lambda^s\leq \mu^s$. Then, $\mathcal{A}^s(\lambda^s)$ is a convex subcategory of $\mathcal{A}^s(\mu^s)$.
\end{proposition}

We may use an example to understand the above proposition. Let $\lambda^s=(4,3,1)$ and $\mu^s=(4,3,2,1)$. Then, $\mathcal{A}^s(\lambda^s)$ is presented by
\begin{center}
$\vcenter{\xymatrix@C=0.7cm@R=0.6cm{\circ\ar[r]& \circ\ar[d]\ar[r]\ar@{.}[dr]& \circ\ar[d]\ar[r]\ar@{.}[dr]&\circ\ar[d]\\
& \circ\ar[r]& \circ\ar[d]\ar[r]&\circ\\
 & &\circ&}}$
\end{center}
with all possible commutativity relations, while $\mathcal{A}^s(\mu^s)$ is presented by the triangle quiver $Q_{tri}$ with all possible commutativity relations. One can easily find that  $\mathcal{A}^s(\lambda^s)$ is a proper convex subcategory of $\mathcal{A}^s(\mu^s)$.

\begin{theorem}\label{rep-type-shifted}
Let $\lambda^s$ be a shifted partition. Then, $\mathcal{A}^s(\lambda^s)$ is representation-finite if and only if $\lambda^s$ is one of $(n)$, $(m-1,1)$ with $m\geqslant 3$, $(3,2)$, $(4,2)$, $(5,2)$, $(6,2)$, $(4,3)$, $(5,3)$, $(5,4)$, $(3,2,1)$ and $(4,2,1)$; tame concealed if and only if $\lambda^s$ is one of $(6,3)$, $(7,2)$ and $(5,2,1)$; tame non-concealed if and only if $\lambda^s$ is one of $(6,4)$, $(6,5)$, $(4,3,1)$, $(4,3,2)$ and $(4,3,2,1)$. Otherwise, $\mathcal{A}^s(\lambda^s)$ is wild.
\end{theorem}
\begin{proof}
We first observe that $\mathcal{A}^s(n)$ and $\mathcal{A}^s(m-1,1)$ with $m\geqslant 3$ are path algebras of Dynkin types $\mathbb{A}_n$ and $\mathbb{D}_m$, respectively. Thus, both of them are representation-finite.

Next, we directly construct the Tits form of $\mathcal{A}^s(\lambda^s)$ for some small cases. Then, we can use Proposition \ref{tits-form} to check their representation type since $\mathcal{A}^s(\lambda^s)$ is strongly simply connected. In particular, we use the software GAP\footnote{One may refer to the commands $\mathsf{IsWeaklyPositiveUnitForm(-)}$ and $\mathsf{IsWeaklyNonnegativeUnitForm(-)}$.} to check whether a Tits form is weakly non-negative (resp., weakly positive) or not, where the method in GAP is introduced by \cite{DP-tits-form-weaklynonnag} (resp., \cite{H-edge-reduction}).

(1) Assume that $Q_{(8,2)}$ is labeled as
\begin{center}
$\xymatrix@C=0.7cm@R=0.6cm{1\ar[r]& 2\ar[r]\ar[d]\ar@{.}[dr]& 3\ar[r]\ar[d]&4\ar[r]&5\ar[r]&6\ar[r]&7\ar[r]&8\\
& 9\ar[r]&10&&&}$.
\end{center}
By the definition of Tits form, we have
\begin{center}
$\begin{aligned}
q_{\mathcal{A}^s(8,2)}(v)= &\ v_1^2-v_1v_2+v_2^2-v_2v_3-v_2v_9+v_2v_{10}+v_3^2\\
&-v_3v_4-v_3v_{10}+v_4^2-v_4v_5+v_5^2-v_5v_6+v_6^2\\
&-v_6v_7+v_7^2-v_7v_8+v_8^2+v_9^2-v_9v_{10}+v_{10}^2.
\end{aligned}$
\end{center}
Then, one can check that
\begin{center}
$q_{\mathcal{A}^s(8,2)}\begin{pmatrix}
\begin{matrix}
10 &20& 25& 20& 16&12&5&1\\
 & 15& 10&  &  &
\end{matrix}\end{pmatrix}=-1$.
\end{center}
We deduce that $\mathcal{A}^s(8,2)$ is wild following Proposition \ref{tits-form}. Then, we observe that $q_{\mathcal{A}^s(7,2)}(v)$ is weakly non-negative, but it is not weakly positive due to
\begin{center}
$q_{\mathcal{A}^s(7,2)}\begin{pmatrix}
\begin{matrix}
2 &4& 5& 4& 3&2&1\\
 & 3& 2&  &  &&
\end{matrix}\end{pmatrix}=0$,
\end{center}
and hence, $\mathcal{A}^s(7,2)$ is tame. Also, one may use GAP to check that $q_{\mathcal{A}^s(6,2)}(v)$ is weakly positive, which implies that $\mathcal{A}^s(6,2)$, $\mathcal{A}^s(5,2)$, $\mathcal{A}^s(4,2)$ and $\mathcal{A}^s(3,2)$ are representation-finite. Similarly, one can see that $\mathcal{A}^s(7,3)$ is wild by
\begin{center}
$q_{\mathcal{A}^s(7,3)}\begin{pmatrix}
\begin{matrix}
2 &4& 6& 6& 4&2&1\\
 & 4& 4& 2&  &&
\end{matrix}\end{pmatrix}=-1$,
\end{center}
and $\mathcal{A}^s(6,3)$ is tame by $q_{\mathcal{A}^s(6,3)}(v)$ being weakly non-negative and
\begin{center}
$q_{\mathcal{A}^s(6,3)}\begin{pmatrix}
\begin{matrix}
1 &2& 3 & 3&2&1\\
 & 2& 2& 1&  &
\end{matrix}\end{pmatrix}=0$.
\end{center}
Also, one may find that $\mathcal{A}^s(5,3)$ and $\mathcal{A}^s(4,3)$ are representation-finite since $q_{\mathcal{A}^s(5,3)}(v)$ is weakly positive.

(2) Assume that $Q_{(6,5)}$ is labeled as
\begin{center}
$\xymatrix@C=0.7cm@R=0.6cm{1\ar[r]& 2\ar[r]\ar[d]\ar@{.}[dr]& 3\ar[r]\ar[d]\ar@{.}[dr]&4\ar[r]\ar[d]\ar@{.}[dr]&5\ar[r]\ar[d]\ar@{.}[dr]&6\ar[d]\\
&7\ar[r]& 8\ar[r]& 9\ar[r]&10\ar[r]&11}$.
\end{center}
Then, the Tits form $q_{\mathcal{A}^s(6,5)}(v)=vXv^{T}$ is given by
\begin{center}
$X=\dfrac{1}{2}\cdot$ \scalebox{0.75}{$ \begin{pmatrix}
2&-1&0&0&0&0&0&0&0&0&0\\
-1&2&-1&0&0&0&-1&1&0&0&0\\
0&-1&2&-1&0&0&0&-1&1&0&0\\
0&0&-1&2&-1&0&0&0&-1&1&0\\
0&0&0&-1&2&-1&0&0&0&-1&1\\
0&0&0&0&-1&2&0&0&0&0&-1\\
0&-1&0&0&0&0&2&-1&0&0&0\\
0&1&-1&0&0&0&-1&2&-1&0&0\\
0&0&1&-1&0&0&0&-1&2&-1&0\\
0&0&0&1&-1&0&0&0&-1&2&-1\\
0&0&0&0&1&-1&0&0&0&-1&2
\end{pmatrix}$}.
\end{center}
It can be checked by GAP that $q_{\mathcal{A}^s(6,5)}(v)$ is weakly non-negative, so that $\mathcal{A}^s(6,5)$ is not wild by Proposition \ref{tits-form}. On the other hand, we have shown in (1) that $\mathcal{A}^s(6,3)$ is tame. Therefore, both $\mathcal{A}^s(6,5)$ and $\mathcal{A}^s(6,4)$ are tame. Also, we find that $\mathcal{A}^s(5,4)$ is representation-finite since the Tits form $q_{\mathcal{A}^s(5,4)}(v)$ is weakly positive.

(3) We point out that $\mathcal{A}^s(6,2,1)$ is wild by
\begin{center}
$q_{\mathcal{A}^s(6,2,1)}\begin{pmatrix}\begin{matrix}
2&4&6&4&2&1\\
&4&4&& \\
&&2&&
\end{matrix}\end{pmatrix}=-1$,
\end{center}
and $\mathcal{A}^s(5,2,1)$ is tame by $q_{\mathcal{A}^s(5,2,1)}(v)$ being weakly non-negative and
\begin{center}
$q_{\mathcal{A}^s(5,2,1)}\begin{pmatrix}
\begin{matrix}
1 &2& 3 & 2&1\\
 & 2& 2& &  \\
 &&1&&
\end{matrix}\end{pmatrix}=0$.
\end{center}
We also find that $\mathcal{A}^s(4,2,1)$ and $\mathcal{A}^s(3,2,1)$ are representation-finite since $q_{\mathcal{A}^s(4,2,1)}(v)$ is weakly positive. Similar to the case $\mathcal{A}^s(6,2,1)$, $\mathcal{A}^s(5,3,1)$ is wild by
\begin{center}
$q_{\mathcal{A}^s(5,3,1)}\begin{pmatrix}\begin{matrix}
1&1&3&4&2\\
&2&4&3& \\
&&2&&
\end{matrix}\end{pmatrix}=-1$.
\end{center}
Similar to the case $\mathcal{A}^s(6,5)$, $q_{\mathcal{A}^s(4,3,2,1)}(v)$ is weakly non-negative and $q_{\mathcal{A}^s(4,3,1)}(v)$ is not weakly positive, so that $\mathcal{A}^s(4,3,2,1)$, $\mathcal{A}^s(4,3,2)$ and $\mathcal{A}^s(4,3,1)$ are tame.

(4) We observe that if $\lambda^s$ does not contain one of $(8,2), (7,3), (6,5), (6,2,1)$ and $(5,3,1)$, then $\lambda^s$ is listed in the above. Thus, all of the remaining cases are wild.

In the following, we may check whether a tame shifted-staircase algebra is concealed or not. Although we can also check that $\mathcal{A}^s(\lambda^s)$ is critical or not by checking the Tits form of all convex subcategories of $\mathcal{A}^s(\lambda^s)$, we would like to trust the list of critical algebras in \cite{B-critical} since the list is also given independently in \cite{HV-tame-concealed}.

We observe that $\mathcal{A}^s(7,2)$ is the critical algebra numbered 18 in \cite{B-critical}, $\mathcal{A}^s(6,3)$ is the critical algebra numbered 93 in \cite{B-critical} and $\mathcal{A}^s(5,2,1)$ is the critical algebra numbered 14 in \cite{B-critical}. Thus, $\mathcal{A}^s(7,2)$, $\mathcal{A}^s(6,3)$ and $\mathcal{A}^s(5,2,1)$ are tame concealed. This indicates that $\mathcal{A}^s(6,5)$ and $\mathcal{A}^s(6,4)$ are not tame concealed because both of them have $\mathcal{A}^s(6,3)$ as a proper convex subcategory. Lastly, since $\mathcal{A}^s(4,3,2,1)$, $\mathcal{A}^s(4,3,2)$ and $\mathcal{A}^s(4,3,1)$ contain the following critical algebra (see \cite[Lemma 3.1]{B-critical}) as a proper convex subcategory,
\begin{center}
$\xymatrix@C=0.5cm@R=0.2cm{\circ\ar[dr]& &\circ\ar[dl]\ar[dr]\ar@{.}[dd]&\\  &\circ\ar[dl]\ar[dr]& &\circ\ar[dl]\\ \circ & &\circ  & }$
\end{center}
they are not tame concealed.
\end{proof}

\begin{corollary}
A shifted-staircase algebra $\mathcal{A}^s(\lambda^s)$ is $\tau$-tilting finite if and only if the shifted partition $\lambda^s$ is one of $(n)$, $(m-1,1)$ with $m\geqslant 3$, $(3,2)$, $(4,2)$, $(5,2)$, $(6,2)$, $(4,3)$, $(5,3)$, $(5,4)$, $(3,2,1)$ and $(4,2,1)$.
\end{corollary}

We define
\begin{center}
$\mathcal{C}:=\left \{ C_n\mid \begin{matrix}
C_n\ \text{is the algebra presented by a triangle quiver with }\\
\frac{n(n+1)}{2}\ \text{vertices and all possible commutativity relations}
\end{matrix}
\right \}$.
\end{center}

\begin{corollary}
Let $C_n\in \mathcal{C}$. Then, the algebra $C_n$ is $\tau$-tilting finite if and only if $n\leqslant 3$.
\end{corollary}
\begin{proof}
It is obvious that $C_n$ is the algebra $\mathcal{A}^s(\lambda^s)$ with $\lambda^s:=(n, n-1, \ldots,2,1)$.
\end{proof}

At the end of this section, we may distinguish the following special cases. Let $\vec{A}_n$ be the path algebra of Dynkin type $\mathbb{A}_n$ associated with linear orientation. We define
\begin{center}
$\mathcal{D}:=\left \{ D_{n}\mid D_{n}\ \text{is the tensor product}\ \vec{A}_{2n+1}\otimes_K \vec{A}_{n} \right \}$
\end{center}
and
\begin{center}
$\mathcal{E}:=\left \{ E_n\mid E_n\ \text{is the Auslander algebra of}\ \vec{A}_{2n}\right \}$.
\end{center}
Then, $D_n$ is the staircase algebra $\mathcal{A}(\lambda)$ with $\lambda=(n^{2n+1})$ and $E_n$ is a quotient algebra of the shifted-staircase algebra $\mathcal{A}^s(\lambda^s)$ with $\lambda^s=(2n, 2n-1, \ldots,1)$, modulo some monomial relations. It is shown in \cite[Corollary 1.13]{La-derived-staircase} that $D_n\in \mathcal{D}$ is derived equivalent to $E_n\in \mathcal{E}$.

We recall the definition of derived equivalence as follows. Let $\mathsf{D^b(mod}\ A)$ be the derived category of bounded complexes of modules from $\mathsf{mod}\ A$, which is the localization of the homotopy category $\mathsf{K^b(proj}\ A)$ with respect to quasi-isomorphisms. Then, $\mathsf{D^b(mod}\ A)$ is a triangulated category. We recall that two algebras $A$ and $B$ are said to be derived equivalent if their derived categories $\mathsf{D^b(mod}\ A)$ and $\mathsf{D^b(mod}\ B)$ are equivalent as triangulated categories.

\begin{remark}
Prof. Ariki has pointed out to me that the derived equivalence between $D_2$ and $E_2$ gives an example that derived equivalence does not necessarily preserve the $\tau$-tilting finiteness. Indeed, Proposition \ref{result-triangular-matrix} implies that $E_2$ is $\tau$-tilting finite because $T_2(\vec{A}_{4})=\vec{A}_{2}\otimes \vec{A}_{4}$ is $\tau$-tilting finite by Corollary \ref{result-rectangle}, while $D_2$ is $\tau$-tilting infinite by Corollary \ref{result-rectangle}.
\end{remark}

\section{An example}
Since we have understood the $\tau$-tilting finiteness of simply connected algebras $A$ by Theorem \ref{result-strongly simply}, the next step should be to explore some combinatorial properties of (support) $\tau$-tilting $A$-modules. In particular, one may ask how many support $\tau$-tilting modules does $A$ have? As an example, we look at the number $\#\mathsf{s\tau\text{-}tilt}\ \Lambda_n$ for a specific $\tau$-tilting finite simply connected algebra $\Lambda_n$.

We denote by $\mathsf{supp}\ M:=\{e_i\mid Me_i\neq 0\}$ the support of a right $A$-module $M$. We say that a right $A$-module $M$ has support-rank $s$ if $\#\mathsf{supp}\ M=s$. We denote by $|M|$ the number of isomorphism classes of indecomposable direct summands of $M$, and by $a_s(A)$ the number of pairwise non-isomorphic basic support $\tau$-tilting $A$-modules with support-rank $s$ for $0\leqslant s\leqslant \left | A \right |$. Since the support $\tau$-tilting $A$-module with support-rank $0$ is unique, we have $a_0(A)=1$. Since each simple $A$-module $S_i$ is an $A/A(1-e_i)A$-module and $A/A(1-e_i)A\simeq K$, the support $\tau$-tilting $A$-modules with support-rank $1$ are exactly the simple $A$-modules. Hence, $a_1(A)=\left | A \right |$. Besides, we have
\begin{center}
$\#\mathsf{s\tau\text{-}tilt}\ A=\sum\limits_{s=0}^{\left | A \right |}a_s(A)$.
\end{center}
Note that $a_{\left | A \right |}(A)$ is just the number of pairwise non-isomorphic basic $\tau$-tilting $A$-modules.

Although the number $\#\mathsf{s\tau\text{-}tilt}\ A$ can be determined for some classes of algebras, such as path algebras of Dynkin type \cite{ONFR-number-tilting}, preprojective algebras of Dynkin type \cite{Mizuno}, etc., it is difficult to find the number $\#\mathsf{s\tau\text{-}tilt}\ A$ for general algebras.

We explain the reasons as follows. If $A$ is a path algebra of Dynkin type, we may find some recursive relations between $a_{s-1}(A)$ and $a_s(A)$ such that we can get a formula to calculate $\#\mathsf{s\tau\text{-}tilt}\ A$. If $A$ is a preprojective algebra of Dynkin type $\Delta$, we may construct a one-to-one correspondence between the (pairwise non-isomorphic basic) support $\tau$-tilting $A$-modules and the elements in the Weyl group $W_\Delta$ associated with $\Delta $. Since the number of elements in $W_\Delta$ is known, we get the number $\#\mathsf{s\tau\text{-}tilt}\ A$.

However, it is difficult to find such a recursive relation or such a one-to-one correspondence in general, so that we can only determine the number $\#\mathsf{s\tau\text{-}tilt}\ A$ by direct computation. We have tried to construct recursive relations between $a_{s-1}(A)$ and $a_s(A)$ for some cases. Here, we present a $\tau$-tilting finite simply connected algebra as an example, and the proof is given in the Appendix A.

\begin{example}\label{example-recursive}
Let $\Lambda_3$ be the path algebra of $\xymatrix@C=0.5cm{\circ&\circ\ar[l]\ar[r]&\circ}$ and $\Lambda_n:=KQ_n/I$ ($n\geqslant 4$) the algebra presented by the following quiver $Q_n$ with $I=\langle\alpha\mu-\beta\nu\rangle$.
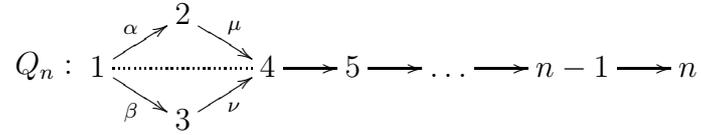
\begin{center}
$Q_n:\vcenter{\xymatrix@C=0.7cm@R=0.2cm{&2\ar[dr]^{\mu}&&&&&\\ 1\ar[ur]^{\alpha}\ar[dr]_{\beta}\ar@{.}[rr]&&4\ar[r]&5\ar[r]&\cdots \ar[r]&n-1\ar[r]&n\\ &3\ar[ur]_{\nu}&&&&&}}$
\end{center}
Then, based on the symbols above, we have
\begin{description}\setlength{\itemsep}{-3pt}
  \item[(1)] $a_n(\Lambda_n)=a_{n-1}(\Lambda_{n})-\frac{1}{n-2}\binom{2n-6}{n-3}$.
  \item[(2)] $a_{n-1}(\Lambda_n)=a_{n-1}(\Lambda_{n-1})+\frac{3n-7}{2n-4}\binom{2n-4}{n-3}+2\cdot \frac{(5n-11)\cdot (2n-6)!}{(n-3)!\cdot (n-1)!}+\sum\limits_{i=4}^{n-1}a_{i-1}(\Lambda_{i-1})\cdot \frac{1}{n-i+1}\binom{2(n-i)}{n-i}$.
  \item[(3)] $a_{n-2}(\Lambda_n)=a_{n-2}(\Lambda_{n-1})+a_{n-3}(\Lambda_n)+\frac{1}{n-2}\binom{2n-6}{n-3}$.
  \item[(4)] $a_s(\Lambda_n)=a_s(\Lambda_{n-1})+a_{s-1}(\Lambda_n)$ for any $1\leqslant s \leqslant n-3$.
\end{description}
These recursive formulas enable us to compute $\#\mathsf{s\tau\text{-}tilt}\ \Lambda_n$ step by step. For example,
\renewcommand\arraystretch{1.2}
\begin{center}
\begin{tabular}{c|cccccccccc|cc}
\diagbox{$n$}{$a_s(\Lambda_n)$}{$s$}&$0$&$1$&$2$&$3$&$4$&$5$&$6$&$7$&8&9&$\#\mathsf{s\tau\text{-}tilt}\ \Lambda_n$  \\ \hline
$4$&1&4&10&16&15&&&&&&46\\
$5$&1&5&15&33&54&52&&&&&160\\
$6$&1&6&21&54&113&192&187&&&&574\\
$7$&1&7&28&82&195&401&700&686&&&2100\\
$8$&1&8&36&118&313&714&1456&2592&2550&&7788\\
$9$&1&9&45&163&476&1190&2646&5307&9702&9570&29172
\end{tabular}.
\end{center}
\end{example}

The following statement implies that $a_s(A)$ is just the number of isomorphism classes of support $\tau$-tilting $A$-modules with $s$ indecomposable direct summands.
\begin{proposition}{\rm{(\cite[Proposition 1.8]{Ada-nakayama})}}\label{support-rank}
Let $M$ be a support $\tau$-tilting $A$-module. Then, $\left | M \right |=\#\mathsf{supp}\ M$.
\end{proposition}

We show that $\Lambda_n$ is a tilted algebra of Dynkin type $\mathbb{D}_n$ and therefore, $\Lambda_n$ is representation-finite following \cite[VIII, Lemma 3.2]{ASS}. Let $\vec{D}_n$ ($n\geqslant 4$) be a path algebra with quiver:
\begin{center}
$\vcenter{\xymatrix@C=0.7cm@R=0.03cm{1\ar[dr]&&&&&\\ &3\ar[r]&4\ar[r]&\cdots \ar[r]&n-1\ar[r]&n \\2\ar[ur]&&&&&}}$.
\end{center}
We denote by $P_i$ (resp., $I_i$) the indecomposable projective (resp., injective) $\vec{D}_n$-modules. Then, by Definition-Theorem \ref{def-left-mutation}, it is easy to find that
\begin{center}
$\mu^-_{P_3}(\vec{D}_n)=P_1\oplus P_2\oplus I_n\oplus P_4\oplus \cdots \oplus P_{n-1}\oplus P_n$,
\end{center}
and $\mu^-_{P_3}(\vec{D}_n)$ is a $\tau$-tilting $\vec{D}_n$-module. Also, $\mu^-_{P_3}(\vec{D}_n)$ is a tilting $\vec{D}_n$-module since tilting and $\tau$-tilting coincide over a path algebra. Then, we have $\Lambda_n=\mathsf{End}_{\vec{D}_n} (\mu^-_{P_3}(\vec{D}_n))$.

\appendix
\section{The proof of Example \ref{example-recursive}.}
For any $M\in \mathsf{mod}\ A$, we denote by $\mathsf{add}(M)$ (resp., $\mathsf{Fac}(M)$) the full subcategory of $\mathsf{mod}\ A$ whose objects are direct summands (resp., factor modules) of finite direct sums of copies of $M$. Moreover, we often describe $A$-modules via their composition series. For example, each simple module $S_i$ is written as $i$, and $\substack{1\\2}$ is an indecomposable $A$-module $M$ with a unique simple submodule $S_2$ such that $M/S_2\simeq S_1$.

Let $(M,P)$ be a pair with $M\in \mathsf{mod}\ A$ and $P\in \mathsf{proj}\ A$. It is called a support $\tau$-tilting (resp., almost complete support $\tau$-tilting) pair if $M$ is $\tau$-rigid, $\mathsf{Hom}_A(P,M)=0$ and $|M|+|P|=|A|$ (resp., $|M|+|P|=|A|-1$). Obviously, $(M,P)$ is a support $\tau$-tilting pair if and only if $M$ is a $\tau$-tilting $\left ( A/Ae A \right )$-module and $P=eA$. It is also useful to recall from \cite[Theorem 2.18]{AIR} that any basic almost complete support $\tau$-tilting pair is a direct summand of exactly two basic support $\tau$-tilting pairs.

We recall the definition of left mutations, which is the core of $\tau$-tilting theory.
\begin{definitiontheorem}(\cite[Definition 2.19, Theorem 2.30]{AIR})\label{def-left-mutation}
Let $T=M\oplus N$ be a basic support $\tau$-tilting $A$-module with an indecomposable direct summand $M$ satisfying $M\notin \mathsf{Fac}(N)$. We take an exact sequence with a minimal left $\mathsf{add}(N)$-approximation $f$:
\begin{center}
$M\overset{f}{\longrightarrow}N'\longrightarrow \mathsf{coker}\ f \longrightarrow 0$.
\end{center}
We call $\mu_M^-(T):=(\mathsf{coker}\ f )\oplus N$ the left mutation of $T$ with respect to $M$, which is again a basic support $\tau$-tilting $A$-module.
\end{definitiontheorem}

We may construct a graph $\mathcal{H}(\mathsf{s\tau\text{-}tilt}\ A)$ by drawing an arrow from $M$ to $N$ if $N$ is a left mutation of $M$. It has been proven that $\mathcal{H}(\mathsf{s\tau\text{-}tilt}\ A)$ is exactly the Hasse quiver on the poset $\mathsf{s\tau\text{-}tilt}\ A$ with respect to a partial order $\leq $. Here, for any $M, N\in \mathsf{s\tau\text{-}tilt}\ A$, $N\leq M$ if $\mathsf{Fac}(N) \subseteq \mathsf{Fac}(M)$. We point out that any $\tau$-tilting finite algebra $A$ admits a finite connected Hasse quiver $\mathcal{H}(\mathsf{s\tau\text{-}tilt}\ A)$.

We denote by $Q_s(A)$ the set of pairwise non-isomorphic basic support $\tau$-tilting $A$-modules with support-rank $s$. In order to give a proof of Example \ref{example-recursive}, we need the following observation. Let ${\bf{A}}_n$ be the path algebra presented by
\begin{center}
$\xymatrix@C=0.7cm{1\ar[r]^{ }&2\ar[r]^{ }&3\ar[r]&\cdots \ar[r]&n-1\ar[r]&n}$.
\end{center}
Then, the indecomposable projective ${\bf{A}}_n$-module $P_1$ at vertex 1 is the unique indecomposable projective-injective ${\bf{A}}_n$-module.
\begin{lemma}{\rm{(\cite[Proposition 2.32, Theorem 2.33]{Ada-nakayama})}}\label{lemma-path-alg-type-A}
With the above notations, any $\tau$-tilting ${\bf{A}}_n$-module $T$ contains $P_1$ as an indecomposable direct summand. Moreover, there exists a bijection
\begin{center}
$Q_n({\bf{A}}_n)\longleftrightarrow Q_{n-1}({\bf{A}}_n)$
\end{center}
given by $Q_n({\bf{A}}_n)\ni T \longmapsto T/P_1 \in Q_{n-1}({\bf{A}}_n)$.
\end{lemma}

It is immediate that $a_n({\bf{A}}_n)=a_{n-1}({\bf{A}}_n)$. This can also be verified by the formula
\begin{center}
$a_s({\bf{A}}_n)=\frac{n-s+1}{n+1}\binom{n+s}{s}$,
\end{center}
which is given by \cite{ONFR-number-tilting}.

In the following, we divide the proof of Example \ref{example-recursive} into several propositions. As a beginning, we find the number $\#\mathsf{s\tau\text{-}tilt}\ \Lambda_4$ by computing the left mutations starting with $\Lambda_4$. Since the number is small, we can do this by hand.
\begin{example}
Let $P_i$ be the indecomposable projective $\Lambda_4$-module at vertex $i$. Then,
\begin{center}
$P_1=\begin{smallmatrix}
 & 1 & \\
2 &  & 3\\
 & 4 &
\end{smallmatrix},
P_2=\begin{smallmatrix}
2\\
4
\end{smallmatrix},
P_3=\begin{smallmatrix}
3\\
4
\end{smallmatrix}$ and $P_4=4$.
\end{center}
By direct calculation, we find that (1) all $\tau$-tilting $\Lambda_4$-modules are
\begin{center}
\renewcommand\arraystretch{1.8}
\begin{tabular}{c|c|c|c|c}
$\substack{\\2\\}\substack{1\\\ \\ 4}\substack{\\3\\}\oplus \substack{2\\4}\oplus \substack{3\\4}\oplus \substack{4}$ & $\substack{\\2\\}\substack{1\\\ \\ 4}\substack{\\3\\}\oplus \substack{2\\4}\oplus \substack{3\\4}\oplus \substack{2\\ \\}\substack{\ \\ \\ 4}\substack{3\\ \\}$ & $\substack{\\2\\}\substack{1\\\ \\ 4}\substack{\\3\\}\oplus \substack{1\\3}\oplus \substack{1\\2}\oplus \substack{\\ \\2}\substack{1\\\ \\}\substack{\\ \\3}$ & $\substack{\\2\\}\substack{1\\\ \\ 4}\substack{\\3\\}\oplus \substack{1\\2}\oplus \substack{2\\4}\oplus \substack{2}$ &$\substack{\\2\\}\substack{1\\\ \\ 4}\substack{\\3\\}\oplus \substack{1\\3}\oplus \substack{1\\2}\oplus \substack{4}$ \\ \hline

$\substack{\\2\\}\substack{1\\\ \\ 4}\substack{\\3\\}\oplus \substack{2\\4}\oplus \substack{1\\2}\oplus \substack{4}$& $\substack{\\2\\}\substack{1\\\ \\ 4}\substack{\\3\\}\oplus \substack{2\\4}\oplus \substack{2}\oplus \substack{2\\ \\}\substack{\ \\ \\ 4}\substack{3\\ \\}$ & $\substack{\\2\\}\substack{1\\\ \\ 4}\substack{\\3\\}\oplus \substack{1\\3}\oplus \substack{3}\oplus \substack{\\ \\2}\substack{1\\\ \\}\substack{\\ \\3}$ & $\substack{\\2\\}\substack{1\\\ \\ 4}\substack{\\3\\}\oplus \substack{1\\3}\oplus \substack{3\\4}\oplus \substack{3}$ &$\substack{\\2\\}\substack{1\\\ \\ 4}\substack{\\3\\}\oplus \substack{3}\oplus \substack{2}\oplus \substack{2\\ \\}\substack{\ \\ \\ 4}\substack{3\\ \\}$ \\ \hline

$\substack{\\2\\}\substack{1\\\ \\ 4}\substack{\\3\\}\oplus \substack{1\\3}\oplus \substack{3\\4}\oplus \substack{4}$ & $\substack{\\2\\}\substack{1\\\ \\ 4}\substack{\\3\\}\oplus \substack{3}\oplus \substack{3\\4}\oplus \substack{2\\ \\}\substack{\ \\ \\ 4}\substack{3\\ \\}$ & $\substack{\\2\\}\substack{1\\\ \\ 4}\substack{\\3\\}\oplus \substack{2}\oplus \substack{1\\2}\oplus \substack{\\ \\2}\substack{1\\\ \\}\substack{\\ \\3}$ &  $\substack{1}\oplus \substack{1\\3}\oplus \substack{1\\2}\oplus \substack{4}$ &$\substack{\\2\\}\substack{1\\\ \\ 4}\substack{\\3\\}\oplus \substack{2}\oplus \substack{3}\oplus \substack{ \\ \\2}\substack{1\\ \\  }\substack{ \\ \\3}$
\end{tabular};
\end{center}
(2) all support $\tau$-tilting $\Lambda_4$-modules with support-rank 3 are
\begin{center}
\renewcommand\arraystretch{1.8}
\begin{tabular}{c|c|c|c|c|c}
$\substack{1\\3}\oplus \substack{1\\2}\oplus \substack{\\ \\2}\substack{1\\\ \\}\substack{\\ \\3}$ & $\substack{2}\oplus \substack{3}\oplus \substack{\\ \\2}\substack{1\\\ \\}\substack{\\ \\3}$ & $\substack{3}\oplus \substack{3\\4}\oplus \substack{2\\ \\}\substack{\ \\ \\ 4}\substack{3\\ \\}$ &$\substack{1\\3}\oplus \substack{3\\4}\oplus \substack{4}$ &$\substack{2\\4}\oplus \substack{1\\2}\oplus \substack{4}$ &   $\substack{1} \oplus \substack{1\\3}\oplus \substack{4}$ \\ \hline

$\substack{1\\3}\oplus \substack{3}\oplus \substack{\\ \\2}\substack{1\\\ \\}\substack{\\ \\3}$ & $\substack{2\\4}\oplus \substack{3\\4}\oplus \substack{2\\ \\}\substack{\ \\ \\ 4}\substack{3\\ \\}$  & $\substack{3}\oplus \substack{2}\oplus \substack{2\\ \\}\substack{\ \\ \\ 4}\substack{3\\ \\}$ & $\substack{2\\4}\oplus \substack{3\\4}\oplus \substack{4}$ & $\substack{2\\4}\oplus \substack{1\\2}\oplus \substack{2}$  & $\substack{1}\oplus \substack{1\\2}\oplus \substack{4}$ \\ \hline

$\substack{2}\oplus \substack{1\\2}\oplus \substack{\\ \\2}\substack{1\\\ \\}\substack{\\ \\3}$ & $\substack{2\\4}\oplus \substack{2}\oplus \substack{2\\ \\}\substack{\ \\ \\ 4}\substack{3\\ \\}$ & $\substack{1}\oplus \substack{1\\3}\oplus \substack{1\\2}$ & $\substack{1\\3}\oplus \substack{3\\4}\oplus \substack{3}$
\end{tabular};
\end{center}
(3) all support $\tau$-tilting $\Lambda_4$-modules with support-rank 2 are
\begin{center}
\renewcommand\arraystretch{1.8}
\begin{tabular}{c|c|c|c|c|c|c|c|c|c|c}
$\substack{3\\4}\oplus \substack{4}$ &
$\substack{2\\4}\oplus \substack{4}$ &
$\substack{1\\2}\oplus \substack{2}$ &
$\substack{1\\3}\oplus \substack{3}$&
$\substack{1}\oplus \substack{4}$&
$\substack{3\\4}\oplus \substack{3}$ &
$\substack{2\\4}\oplus \substack{2}$  &
$\substack{1\\2}\oplus \substack{1}$  &
$\substack{1\\3}\oplus \substack{1}$ &
$\substack{2}\oplus \substack{3}$
\end{tabular}.
\end{center}
Combining the above with $a_0(\Lambda_4)=1$ and $a_1(\Lambda_4)=4$, we conclude that $\#\mathsf{s\tau\text{-}tilt}\ \Lambda_4=46$.
\end{example}

\begin{proposition}\label{lambda-s}
For any $n\geqslant 4$ and $1\leqslant s \leqslant n-3$, we have
\begin{center}
$a_s(\Lambda_n)=a_s(\Lambda_{n-1})+a_{s-1}(\Lambda_n)$.
\end{center}
\end{proposition}
\begin{proof}
Let $e_i$ be the idempotent of $\Lambda_n$ at vertex $i$. For any support $\tau$-tilting $\Lambda_n$-module $M$ satisfying $Me_n=0$, it is obvious that $M$ is a support $\tau$-tilting $\Lambda_{n-1}$-module. Then, let $Q_s(\Lambda_n; e_n)$ be the set of the support $\tau$-tilting $\Lambda_n$-modules $T$ with support-rank $s$ and $Te_n\neq0$. We show that there is a bijection
\begin{center}
$\mathfrak{q}: Q_s(\Lambda_n; e_n)\longrightarrow Q_{s-1}(\Lambda_n)$,
\end{center}
and then, the statement follows from this bijection.

Let $X$ be an indecomposable $\Lambda_n$-module with support-rank $t \leqslant n-3$ and $Xe_n\neq0$. Then, $X$ is an indecomposable module over a path algebra of type $\mathbb{A}$ and it corresponds to a root so that $X$ is of the form
\begin{center}
$\substack{S_{n-t+1}\\ \vdots\\ S_{n-1}\\ S_n}$.
\end{center}
In this case, we denote $X$ by $\left [ n-t+1, n \right ]$.

Let $T\in Q_s(\Lambda_n; e_n)$ and $(T,P)$ the corresponding support $\tau$-tilting pair. There exists at least one indecomposable direct summand of $T$, say $X$, which satisfies $Xe_n\neq0$ and we choose $X:=\left [ n-t+1, n \right ]$ of the largest possible length $t$.
\begin{description}\setlength{\itemsep}{-3pt}
  \item[(1)] We show that $Te_m=0$ for any arrow $m\longrightarrow n-t+1$. In fact, if $t=s$, it is true since $T$ is support-rank $s$. If $t\leqslant s-1$, the inequality $n-t+1\geqslant 5$ makes $m$ unique and $m=n-t$. By the maximality of $X$, the number of indecomposable direct summands $X'$ of $T$ with $X'e_n\neq0$ is at most $t$ and $t+|P|=t+n-s\leqslant n-1$. (Note that $X'e_{n-t}=0$ is obvious.) We consider the remaining indecomposable direct summand $Y$ of $T$ satisfying $Ye_n=0$. Suppose that $Ye_{n-t}\neq 0$. It is enough to consider the following five types of $Y$:
\begin{center}
$\begin{smallmatrix}
&S_1&\\
S_2&&S_3\\
&S_4&\\
&\vdots &\\
&S_{n-t}&\\
&\vdots &\\
&S_a&
\end{smallmatrix}$, $\begin{smallmatrix}
S_2&&S_3\\
&S_4&\\
&\vdots &\\
&S_{n-t}&\\
&\vdots &\\
&S_a&
\end{smallmatrix}$, $\begin{smallmatrix}
S_2\\
S_4\\
\vdots \\
S_{n-t}\\
\vdots\\
S_a
\end{smallmatrix}$, $\begin{smallmatrix}
S_3\\
S_4\\
\vdots \\
S_{n-t}\\
\vdots\\
S_a
\end{smallmatrix}$, $\begin{smallmatrix}
S_4\\
\vdots \\
S_{n-t}\\
\vdots\\
S_a
\end{smallmatrix}$,
\end{center}
where $4\leqslant n-t\leqslant a\leqslant n-1$. One can check that $\left [ n-t+1, a+1 \right ]$ is a submodule of $\tau Y$ for any type above. Then, $\mathsf{Hom}_{\Lambda_n}(X, \tau Y)\neq 0$ and it contradicts with $T\in Q_s(\Lambda_n; e_n)$. Therefore, we must have $Ye_{n-t}=0$.

\item[(2)] According to (1), we can divide $T$ into a direct sum $W\oplus Z$ such that $\mathsf{supp}\ W=\left \{ e_{n-t+1},\dots,e_{n-1},e_n \right \}$, $\mathsf{supp}\ W\cap \mathsf{supp}\ Z=\varnothing $ and $\mathsf{supp}\ Z$ does not contain $e_m$ with $m\longrightarrow n-t+1$. Then, we define
\begin{center}
$\Lambda_{[n-t+1, n]}:=\Lambda_n/\langle e_1+e_2+\cdots +e_{n-t}\rangle$.
\end{center}
Since $T\in Q_s(\Lambda_n; e_n)$ and the supports of $W$ and $Z$ are disjoint, $W$ is actually a support $\tau$-tilting $\Lambda_n$-module by repeatedly calculating the left mutations started at direct summands of $Z$. By Proposition \ref{support-rank}, $|W|=t$ since the support-rank of $W$ is $t$. Then, $W$ becomes a $\tau$-tilting $\Lambda_{[n-t+1, n]}$-module. We note that $X$ is the unique indecomposable projective-injective $\Lambda_{[n-t+1, n]}$-module since $\Lambda_{[n-t+1, n]}$ is isomorphic to the path algebra ${\bf{A}}_t$. By Lemma \ref{lemma-path-alg-type-A}, the quotient module $W/X$ is a support $\tau$-tilting $\Lambda_{[n-t+1, n]}$-module with support-rank $t-1$.

\item[(3)] Based on the analysis in (1) and (2), we define $U:=T/X$ for any $T\in Q_s(\Lambda_n; e_n)$. Since $T$ is a support $\tau$-tilting $\Lambda_n$-module with support-rank $s$, $U$ is a support $\tau$-tilting $\Lambda_n$-module with support-rank $s-1$ by $W/X\in Q_{t-1}(\Lambda_{[n-t+1, n]})$ and the fact that the supports of $W$ and $Z$ are disjoint. Thus, we may define the map from $Q_s(\Lambda_n; e_n)$ to $Q_{s-1}(\Lambda_n)$ by $\mathfrak{q}(T)=U$.
\end{description}

Next, we show that the map $\mathfrak{q}$ defined above is a bijection. On the one hand, we know that $\mathfrak{q}$ is an injection. By the analysis in $(2)$, we may define $T_1:=Z_1\oplus X_1\oplus V_1\in Q_s(\Lambda_n; e_n)$ and $T_2:=Z_2\oplus X_2\oplus V_2\in Q_s(\Lambda_n; e_n)$ such that
\begin{itemize}\setlength{\itemsep}{-3pt}
  \item $X_i=[n-t_i+1, n]$ for $i=1,2$,
  \item $\mathsf{supp}\ Z_1$ is included in $\{e_1,e_2,\ldots, e_{n-t_1-1}\}$,
  \item $\mathsf{supp}\ Z_2$ is included in $\{e_1,e_2,\ldots, e_{n-t_2-1}\}$,
  \item $\mathsf{supp}\ V_1=\{e_{n-t_1+1},e_{n-t_1+2},\ldots,e_{i_1-1}, e_{i_1+1},\ldots, e_n\}$ with exactly one $e_{i_1}$ satisfying $V_1e_{i_1}=0$ for $n-t_1+1\leqslant i_1\leqslant n$,
  \item $\mathsf{supp}\ V_2=\{e_{n-t_2+1},e_{n-t_2+2},\ldots,e_{i_2-1}, e_{i_2+1},\ldots, e_n\}$ with exactly one $e_{i_2}$ satisfying $V_2e_{i_2}=0$ for $n-t_2+1\leqslant i_2\leqslant n$.
\end{itemize}
Obviously, $X_1\neq X_2$ if and only if $t_1\neq t_2$. If $X_1=X_2$, $T_1\neq T_2$ implies $Z_1\oplus V_1\neq Z_2\oplus V_2$. Then, we assume that $X_1\neq X_2$. If we list the idempotents by increasing the subscripts, the last two idempotents outside of $\mathsf{supp}\ (Z_1\oplus V_1)$ must be $\{e_{n-t_1},e_{i_1}\}$ and the last two idempotents outside of $\mathsf{supp}\ (Z_2\oplus V_2)$ must be $\{e_{n-t_2},e_{i_2}\}$. Since $t_1\neq t_2$, we have $e_{n-t_1}\neq e_{n-t_2}$ so that $\mathsf{supp}\ (Z_1\oplus V_1)$ and $\mathsf{supp}\ (Z_2\oplus V_2)$ are different. Thus, we conclude that $\mathfrak{q}(T_1)\neq \mathfrak{q}(T_2)$ if $T_1\neq T_2\in Q_s(\Lambda_n; e_n)$.

On the other hand, $\mathfrak{q}$ is a surjection. We assume that $U\in Q_{s-1}(\Lambda_n)$. Since $s-1\leqslant n-4$, there are at least 4 idempotents of $\Lambda_n$ outside of $\mathsf{supp}\ U$.
\begin{itemize}\setlength{\itemsep}{-3pt}
  \item If there are exactly 4 idempotents $e_1,e_2,e_3$ and $e_i$ with $4\leqslant i\leqslant n$ outside of $\mathsf{supp}\ U$, then $s=n-3$ and $U$ becomes a support $\tau$-tilting $\Lambda_{[4, n]}$-module with support-rank $n-4$. Let $P:=[4,n]$ be the indecomposable projective $\Lambda_{[4, n]}$-module at vertex $4$. Since $\Lambda_{[4, n]}$ is isomorphic to the path algebra ${\bf{A}}_{n-3}$, we have $T:=P\oplus U\in Q_{n-3}(\Lambda_n; e_n)$ by Lemma \ref{lemma-path-alg-type-A}, and $T$ maps to $U$.
  \item Otherwise, there are at least two idempotents in $\left \{ e_4,e_5,\dots,e_n \right \}$ outside of $\mathsf{supp}\ U$. Let $j>i>\cdots \geqslant 4$ be the first two numbers in decreasing order of such subscripts. Then, $e_{i+1}$ does not appear in $\mathsf{supp}\ \tau U$, and we can find an indecomposable projective $\Lambda_n$-module $P:=[i+1,n]$ such that $\mathsf{Hom}_{\Lambda_n}(P, \tau U)= 0$. Hence, $T:=P\oplus U \in Q_s(\Lambda_n; e_n)$. In fact, $Te_j\neq 0$ and $Te_{j'}=0$ for any $j'\neq j$ satisfying $Ue_{j'}=0$. Moreover, $T$ maps to $U$.
\end{itemize}
Therefore, $\mathfrak{q}$ is a surjection.
\end{proof}

We denote by $[2,4,\ldots,n]$ (resp., $[3,4,\ldots,n]$) the indecomposable projective $\Lambda_n$-module at vertex 2 (resp., 3).
\begin{proposition}\label{lambda-n-2}
For any $n\geqslant 4$, we have
\begin{center}
$a_{n-2}(\Lambda_n)=a_{n-2}(\Lambda_{n-1})+a_{n-3}(\Lambda_n)+a_{n-3}({\bf{A}}_{n-3})$.
\end{center}
\end{proposition}
\begin{proof}
Let $e_i$ be the idempotent of $\Lambda_n$ at vertex $i$. For any support $\tau$-tilting $\Lambda_n$-module $M$ satisfying $Me_n=0$, it is obvious that $M$ is a support $\tau$-tilting $\Lambda_{n-1}$-module. Then, the number of such modules with support-rank $n-2$ is $a_{n-2}(\Lambda_{n-1})$.

Let $Q_{n-2}(\Lambda_n; e_n)$ be the set of support $\tau$-tilting $\Lambda_n$-modules $T$ with support-rank $n-2$ and $Te_n\neq0$. For any $T\in Q_{n-2}(\Lambda_n; e_n)$, we denote it by $T=X\oplus U$ with an indecomposable direct summand $X$ satisfying $Xe_n\neq0$. We may set $X:=\left [n-t+1, n \right ]$ of the largest possible length $t$, while $X=[2,4,\ldots,n]$ is also allowed if $t=n-2$. We show that $\mu^-_X(T)=U$ and therefore, $U\in Q_{n-3}(\Lambda_n)$. Then, we can define a map $\mathfrak{q}$ from $Q_{n-2}(\Lambda_n; e_n)$ to $Q_{n-3}(\Lambda_n)$ by $\mathfrak{q}(T)=U$.
\begin{itemize}\setlength{\itemsep}{-3pt}
  \item If $t\leqslant n-3$, this is similar to the situation in the proof of Proposition \ref{lambda-s}. Thus, $Te_m=0$ for any arrow $m\longrightarrow n-t+1$ such that $T=X\oplus V\oplus Z$, where $X\oplus V$ is a $\tau$-tilting $\Lambda_{[n-t+1, n]}$-module with the unique indecomposable projective-injective $\Lambda_{[n-t+1, n]}$-module $X$, the supports of $X\oplus V$ and $Z$ are disjoint and $\mathsf{supp}\ Z$ does not contain $e_m$ with $m\longrightarrow n-t+1$. By Lemma \ref{lemma-path-alg-type-A}, we have $\mu^-_X(T)=V\oplus Z$. In this case, $\mathfrak{q}(T_1)\neq \mathfrak{q}(T_2)$ if $T_1\neq T_2 \in Q_{n-2}(\Lambda_n; e_n)$.
  \item Let $t=n-2$. Then, $\mathsf{supp}\ X$ is either $\left \{ e_2,e_4,\dots,e_n \right \}$ or $\left \{ e_3,e_4,\dots,e_n \right \}$ such that $X$ is uniquely determined (since $\mathsf{supp}\ X$ cannot contain all the idempotents $\left \{ e_2,e_3, e_4,\dots,e_n \right \}$). Then, $Te_1=0$ is obvious and $\mu^-_X(T)=U$ is also true. In fact, let $T:=X\oplus U$ and $X:=[2,4,\ldots,n]$. The condition $Te_1=Te_3=0$ makes $T$ to be a $\tau$-tilting ${\bf{A}}_{n-2}$-module and makes $X$ to be the unique indecomposable projective-injective ${\bf{A}}_{n-2}$-module. By Definition-Theorem \ref{def-left-mutation}, we deduce that $\mu^-_X(T)=U$. Similarly, one can observe the fact $\mu^-_X(T)=U$ for $X=[3,4,\ldots,n]$.
\end{itemize}

Let $X_1:=[2,4,\ldots,n]$ and $X_2:=[3,4,\ldots,n]$. We observe that the map $\mathfrak{q}$ defined above is not an injection because $\mathfrak{q}(X_1\oplus U)=\mathfrak{q}(X_2\oplus U)=U$ whenever $U$ is a $\tau$-tilting $\Lambda_{[4,n]}$-module, which appears in the case $t=n-2$. Also, it will be useful to mention that for any $T\in Q_{n-2}(\Lambda_n; e_n)$, if there exists an idempotent $e_i\in\{e_4,e_5,\ldots, e_n\}$ satisfying $Te_i=0$, then $t\leqslant n-i\leqslant n-4$.

Next, we show that the map $\mathfrak{q}: Q_{n-2}(\Lambda_n; e_n)\rightarrow Q_{n-3}(\Lambda_n)$ is a surjection. For any $U\in Q_{n-3}(\Lambda_n)$, there exist exactly three idempotents of $\Lambda_n$ outside of $\mathsf{supp}\ U$.
\begin{itemize}\setlength{\itemsep}{-3pt}
  \item Suppose that there are at least two idempotents in $\left \{ e_4,e_5,\dots,e_n \right \}$ outside of $\mathsf{supp}\ U$. Let $j>i\geqslant 4$ (or $j>i>k\geqslant 4$ if $k$ exists) be the order of such numbers. Then, $e_{i+1}$ does not appear in $\mathsf{supp}\ \tau U$, so that $\mathsf{Hom}_{\Lambda_n}(P, \tau U)= 0$ with the indecomposable projective $\Lambda_n$-module $P:=[i+1,n]$. Hence, $P\oplus U \in Q_{n-2}(\Lambda_n; e_n)$ maps to $U$.

      Note that for any $Y$ satisfying $Y\oplus U\in Q_{n-2}(\Lambda_n; e_n)$ and $\mathfrak{q}(Y\oplus U)=U$, the support-rank of $Y\oplus U$ is the support-rank of $U$ plus 1. Then, there always exists an idempotent $e_k\in\{e_4,e_5,\ldots, e_n\}$ satisfying $(Y\oplus U)e_k=0$, so that $Y\oplus U$ is included in the case $t\leqslant n-4$. By the injectivity of $\mathfrak{q}$ in the case $t\leqslant n-3$, we conclude that $P\oplus U \in Q_{n-2}(\Lambda_n; e_n)$ is uniquely determined in this case.

  \item Suppose that there is exactly one idempotent $e_i$ in $\left \{ e_4,e_5,\dots,e_n \right \}$ outside of $\mathsf{supp}\ U$, i.e., there are exactly two idempotents of $e_1,e_2,e_3$ outside of $\mathsf{supp}\ U$.

      If $Ue_1\neq 0, Ue_2=0, Ue_3=0$, then $U=S_1\oplus V$, where $V$ is a support $\tau$-tilting $\Lambda_{[4, n]}$-module with support-rank $n-4$. We observe that $(U, P_2\oplus P_3)$ is an almost complete support $\tau$-tilting pair, so that it has two completions and one of which is $(U, P_2\oplus P_3\oplus P_i)$. Since $Ue_j\neq 0$ for any $i\neq j\geqslant 4$, $(U, P_2\oplus P_3\oplus P_j)$ cannot be a support $\tau$-tilting pair and then, the other completion must be of the form $(U\oplus Y, P_2\oplus P_3)$. In particular, $Ye_2=Ye_3=0$ holds. Since $Y$ is indecomposable and cannot be $S_1$, $Ye_1=0$ also holds. Since the support-rank of $Y\oplus U$ is the support-rank of $U$ plus 1, we must have $Ye_i\neq 0$. Then, $Y\oplus V$ is a $\tau$-tilting $\Lambda_{[4, n]}$-module. Since $V$ is support-rank $n-4$, it cannot have $[4,n]$ as a direct summand. By Lemma \ref{lemma-path-alg-type-A}, $Y=[4,n]$ is unique and $Y\oplus U\in Q_{n-2}(\Lambda_n; e_n)$. Moreover, $\mathfrak{q}(Y\oplus U)=U$.

      If $Ue_1=0, Ue_2\neq0, Ue_3=0$, similar to the above case, we can find a $Y$ satisfying $Ye_1=Ye_3=0$ such that $Y\oplus U\in Q_{n-2}(\Lambda_n; e_n)$ maps to $U$. Then, $U$ can be considered as a support $\tau$-tilting ${\bf{A}}_{n-2}$-module with support-rank $n-3$ as well as $Y\oplus U$ can be considered as a $\tau$-tilting ${\bf{A}}_{n-2}$-module. By Lemma \ref{lemma-path-alg-type-A}, $Y$ is unique and it must be $X_1$.

      If $Ue_1=0, Ue_2=0, Ue_3\neq 0$, similar to above, $X_2\oplus U\in Q_{n-2}(\Lambda_n; e_n)$ maps to $U$ and $X_2$ is uniquely determined by Lemma \ref{lemma-path-alg-type-A}.

   \item Suppose that $e_1,e_2,e_3$ are outside of $\mathsf{supp}\ U$. Then, $U$ is actually a $\tau$-tilting $\Lambda_{[4,n]}$-module and $[4,n]$ must be an indecomposable (projective) direct summand of $U$ following Lemma \ref{lemma-path-alg-type-A}. Then, $\mathsf{Hom}_{\Lambda_n}(X_i, \tau U)= 0$ for $i=1,2$ implies that $X_1\oplus U, X_2\oplus U \in Q_{n-2}(\Lambda_n; e_n)$ and both of them are mapped to $U$.

       We show that $X_1$ and $X_2$ are the only possible cases. Assume that $Y\oplus U \in Q_{n-2}(\Lambda_n; e_n)$ maps to $U$. Then, only one of $Ye_1$, $Ye_2$, $Ye_3$ is not zero. If $Ye_1\neq 0$, $Y$ must be $S_1$ since $Y$ is indecomposable and $Ye_2=Ye_3=0$. But, this contradicts with the fact $Ye_n\neq 0$ deduced by $\mathfrak{q}(Y\oplus U)=U$. We have either $Ye_1=Ye_2=0$, $Ye_3\neq 0$, $Ye_n\neq 0$ or $Ye_1=Ye_3=0$, $Ye_2\neq 0$, $Ye_n\neq 0$, so that $Y=X_1$ or $X_2$.
\end{itemize}

Now, we found that the map $\mathfrak{q}: Q_{n-2}(\Lambda_n; e_n)\rightarrow Q_{n-3}(\Lambda_n)$ is indeed a surjection. If one wants to use this surjection to count the number of modules in $Q_{n-2}(\Lambda_n; e_n)$, one should note that both $X_1\oplus U$ and $X_2\oplus U $ are mapped to $U$ whenever $U$ is a $\tau$-tilting $\Lambda_{[4,n]}$-module, and it is the only part that needs to be double calculated. These are exactly the $a_{n-3}(\Lambda_{[4,n]})=a_{n-3}({\bf{A}}_{n-3})$ pairs of modules in $Q_{n-2}(\Lambda_n; e_n)$. Therefore, the number of modules in $Q_{n-2}(\Lambda_n; e_n)$ is $a_{n-3}(\Lambda_n)+a_{n-3}({\bf{A}}_{n-3})$.
\end{proof}

We define $\mathbb{A}_{2}^{1}:={\bf{A}}_2$ and $\mathbb{A}_{n}^{1}:=KQ/\langle\alpha\beta\rangle$ for any $n\geqslant3$, where
\begin{center}
$ Q: \xymatrix{1\ar[r]^{\alpha}&2\ar[r]^{\beta}&3\ar[r]&\cdots \ar[r]&n-1\ar[r]&n}$.
\end{center}

\begin{proposition}\label{A.6}
Let $n\geqslant 4$, we have
\begin{description}\setlength{\itemsep}{-3pt}
  \item[(1)] $a_s(\mathbb{A}_{n}^{1})=a_s(\mathbb{A}_{n-1}^{1})+a_{s-1}(\mathbb{A}_{n}^{1})$ for any $1\leqslant s \leqslant n-2$.
  \item[(2)]  $a_{n-1}(\mathbb{A}_{n}^{1})=a_{n-1}(\mathbb{A}_{n-1}^{1})+a_{n-1}({\bf{A}}_{n-1})+a_{n-2}({\bf{A}}_{n-2})+\sum\limits_{i=3}^{n-1}a_{i-1}(\mathbb{A}_{i-1}^{1})\cdot a_{n-i}({\bf{A}}_{n-i})$.
  \item[(3)]  $a_n(\mathbb{A}_{n}^{1})=a_{n-1}({\bf{A}}_{n-1})+a_{n-2}({\bf{A}}_{n-2})$.
\end{description}
\end{proposition}
\begin{proof}
(1) The proof is similar to the proof of Proposition \ref{lambda-s}. We omit the details.

(2) By the definition of support $\tau$-tilting modules, we have
\begin{center}
$a_{n-1}(\mathbb{A}_{n}^{1})=\sum_{i=1}^{n}a_{n-1}(\mathbb{A}_{n}^{1}/\langle e_i\rangle)$.
\end{center}
Then, one can easily check the formula.

(3) This follows from Proposition 2.32 and Theorem 2.33 in \cite{Ada-nakayama}.
\end{proof}

By combining Proposition \ref{A.6} (3) and $a_n({\bf{A}}_n)=\frac{1}{n+1}\binom{2n}{n}$, we have
\begin{corollary}\label{A1}
$a_n(\mathbb{A}_{n}^{1})=\frac{1}{n}\binom{2n-2}{n-1}+\frac{1}{n-1}\binom{2n-4}{n-2}$ is the sequence \underline{A005807} in \cite{S-oeis}.
\end{corollary}

Let ${\bf{D}}_n$ be the path algebra presented by
\begin{center}
$\vcenter{\xymatrix@C=0.7cm@R=0.01cm{1\ar[dr]^{ }&&&&&\\ &3\ar[r]&4\ar[r]&\cdots \ar[r]&n-1\ar[r]&n\\2\ar[ur]_{ }&&&&&}}$.
\end{center}
We recall $a_{n}(\textbf{D}_n)=\frac{3n-4}{2n-2}\binom{2n-2}{n-2}$ from \cite{ONFR-number-tilting}.

\begin{proposition}\label{lambda-n-1}
For any $n\geqslant 5$, we have
\begin{center}
$a_{n-1}(\Lambda_n)=a_{n-1}(\Lambda_{n-1})+a_{n-1}({\bf{D}}_{n-1})+2a_{n-1}(\mathbb{A}_{n-1}^{1})+\sum\limits_{i=4}^{n-1}a_{i-1}(\Lambda_{i-1})\cdot a_{n-i}({\bf{A}}_{n-i})$.
\end{center}
\end{proposition}
\begin{proof}
This is similar to Proposition \ref{A.6} (2). We omit the details.
\end{proof}

\begin{proposition}\label{lambda-n}
For any $n\geqslant 4$, we have
\begin{center}
$a_n(\Lambda_n)=a_{n-1}(\Lambda_{n})-a_{n-3}({\bf{A}}_{n-3})$.
\end{center}
\end{proposition}
\begin{proof}
Let $P_i$ be the indecomposable projective $\Lambda_n$-module at vertex $i$. We explain the relation between $Q_n(\Lambda_n)$ and $Q_{n-1}(\Lambda_n)$ as follows.

(1) We show that a $\tau$-tilting $\Lambda_n$-module either has $P_1$ as a direct summand or is of form $\substack{1}\oplus \substack{1\\2}\oplus \substack{1\\3}\oplus V$ with $V$ to be a $\tau$-tilting $\Lambda_{[4,n]}$-module. By the poset structure on $\mathsf{s\tau\text{-}tilt}\ \Lambda_n$, any $\tau$-tilting $\Lambda_n$-module $M\not \simeq \Lambda_n$ satisfies $M\leq \mu_{P_i}^-(\Lambda_n)$ for some $1\leqslant i\leqslant n$. Since $\mu_{P_1}^-(\Lambda_n)$ is not $\tau$-tilting and all $\mu_{P_i}^-(\Lambda_n)$ with $i\geqslant 2$ have $P_1$ as a direct summand, if $M$ does not have $P_1$ as a direct summand, then there must exist a $\tau$-tilting $\Lambda_n$-module $T:=P_1\oplus U$ such that $M\leq \mu_{P_1}^-(T)$. Immediately, we have

\textbf{Case (a).} If $T=P_1\oplus U$ with $Ue_1=0$, then $\mu_{P_1}^-(T)=U$ by Definition-Theorem \ref{def-left-mutation} and $U\in Q_{n-1}(\Lambda_n)$ by Proposition \ref{support-rank}.

Suppose that $Ue_1\neq 0$. We remark that $U$ does not have $S_1$ as a direct summand since $\mathsf{Hom}_{\Lambda_n}(P_1, \tau S_1)\neq 0$. We define
\begin{center}
$M_a:=\begin{smallmatrix}
&1&\\
2&&3\\
&4&\\
&\vdots &\\
&a&
\end{smallmatrix}$
\end{center}
with $4\leqslant a\leqslant n-1$. Then, $U$ has at least one of $\substack{1\\2}$, $\substack{1\\3}$, $\substack{\\ \\2}\substack{1\\ \\ }\substack{\\ \\3}$ and $M_a$ as a direct summand, because these modules are $\tau$-rigid and $Ue_1\neq 0$. In particular, it is worth mentioning that the unique non-zero morphism $f: P_1\rightarrow X$ for any $X\in \{\substack{1\\2}, \substack{1\\3}, \substack{\\ \\2}\substack{1\\ \\ }\substack{\\ \\3},M_a\}$, is actually the projective cover of $X$ and $\mathsf{coker}\ f =0$.
\begin{description}\setlength{\itemsep}{-3pt}
 \item[(a1)] If $U=\substack{1\\2}\oplus V$ with $Ve_1=0$, we have $\mu_{P_1}^-(T)=U$ by substituting $\mathsf{coker}\ f =0$ into Definition-Theorem \ref{def-left-mutation} and hence, $U\in Q_{n-1}(\Lambda_n)$. Since $\tau(\substack{1\\2})=S_3$, we have $S_3\not\subseteq \mathsf{top}\ V$ so that $Ve_3=0$. This implies that $Ue_3=0$ and $Ue_i\neq 0$ for any $i\neq 3$.

  \item[(a2)] If $U=\substack{1\\3}\oplus V$ with $Ve_1=0$, we also have $\mu_{P_1}^-(T)=U$ and $U\in Q_{n-1}(\Lambda_n)$. Since $\tau(\substack{1\\3})=S_2$, we have $Ue_2=0$ and $Ue_i\neq 0$ for any $i\neq 2$.

  \item[(a3)] If $U$ has exactly one of $\substack{\\ \\2}\substack{1\\ \\ }\substack{\\ \\3}$ and $M_a$ as a direct summand, we have $\mu_{P_1}^-(T)=U$ by substituting $\mathsf{coker}\ f =0$ into Definition-Theorem \ref{def-left-mutation} and hence, $U\in Q_{n-1}(\Lambda_n)$. This implies that $Ue_i=0$ for exactly one $i$ with $4\leqslant i\leqslant n$.

  \item[(a4)] If $U$ has two of $\substack{1\\2}$, $\substack{1\\3}$, $\substack{\\ \\2}\substack{1\\ \\ }\substack{\\ \\3}$ and $M_a$ as direct summands and one of the direct summands is $\substack{\\ \\2}\substack{1\\ \\ }\substack{\\ \\3}$ or $M_a$, we also have $\mu_{P_1}^-(T)=U$ since there exist epimorphisms from $\substack{\\ \\2}\substack{1\\ \\ }\substack{\\ \\3}$ or $M_a$ to $\substack{1\\2}$ and $\substack{1\\3}$. Then, $U\in Q_{n-1}(\Lambda_n)$. Since we can make sure that $Ue_2\neq0$ and $Ue_3\neq0$, we have $Ue_i=0$ for exactly one $i$ with $4\leqslant i\leqslant n$.
  \item[(a5)] If $U$ has more than two of $\substack{1\\2}$, $\substack{1\\3}$, $\substack{\\ \\2}\substack{1\\ \\ }\substack{\\ \\3}$ and $M_a$ as direct summands, it must have $\substack{\\ \\2}\substack{1\\ \\ }\substack{\\ \\3}$ or $M_a$ as a direct summand. Then, similar to the above, we have $\mu_{P_1}^-(T)=U$ and $U\in Q_{n-1}(\Lambda_n)$. This also implies that $Ue_i=0$ for exactly one $i$ with $4\leqslant i\leqslant n$.
\end{description}
Then, it remains to consider

\textbf{Case (b).} If $T=P_1\oplus \substack{1\\2}\oplus \substack{1\\3}\oplus V$ with $Ve_1=0$, then $Ve_2=Ve_3=0$, so that $V$ becomes a $\tau$-tilting $\Lambda_{[4,n]}$-module. In fact, if $Ve_2\neq 0$, it implies $S_2\subseteq \mathsf{top}\ V$ so that $\mathsf{Hom}_{\Lambda_n}(V, S_2)\neq 0$. Then, $\tau(\substack{1\\3})=S_2$ implies that $Ve_2$ must be zero. Similarly, $\tau(\substack{1\\2})=S_3$ makes $Ve_3$ to be zero. Then, we have $\mu_{P_1}^-(T)=\substack{1}\oplus \substack{1\\2}\oplus \substack{1\\3}\oplus V$ by simple observation.

Now, we can claim that if a $\tau$-tilting $\Lambda_n$-module $M$ does not have $P_1$ as a direct summand, $M\leq T':=\substack{1}\oplus \substack{1\\2}\oplus \substack{1\\3}\oplus V$ with a $\tau$-tilting $\Lambda_{[4,n]}$-module $V$. Then, we observe that the left mutations of $T'$ with respect to $\substack{1\\2}$ and $\substack{1\\3}$ are not $\tau$-tilting, and the left mutation of $T'$ with respect to one of direct summands of $V$ is equivalent to that of a $\tau$-tilting $\Lambda_{[4,n]}$-module $V$. Therefore, $\substack{1}\oplus \substack{1\\2}\oplus \substack{1\\3}$ must remain in $M$ as a direct summand.

Finally, we conclude that if a $\tau$-tilting $\Lambda_n$-module $M$ does not have $P_1$ as a direct summand, it is of form $\substack{1}\oplus \substack{1\\2}\oplus \substack{1\\3}\oplus V$ with a $\tau$-tilting $\Lambda_{[4,n]}$-module $V$. Moreover, if a $\tau$-tilting $\Lambda_n$-module $T$ has $P_1$ as a direct summand, then $T/P_1\in Q_{n-1}(\Lambda_n)$ if and only if $T\not\simeq P_1\oplus \substack{1\\2}\oplus \substack{1\\3}\oplus V$ with a $\tau$-tilting $\Lambda_{[4,n]}$-module $V$. (This implies that $T/P_1$ is not always included in $Q_{n-1}(\Lambda_n)$. This is also the reason why we distinguish the following set $\mathcal{S}$.)

(2) We may construct a map $\mathfrak{q}$ from $Q_{n-1}(\Lambda_n)$ to $Q_n(\Lambda_n)\setminus  \mathcal{S}$, where
\begin{center}
$\mathcal{S}:=\left \{ P_1\oplus \substack{1\\2}\oplus \substack{1\\3}\oplus V \mid V \ \text{is a}\ \tau\text{-tilting}\  \Lambda_{[4,n]}\text{-module}   \right \}$.
\end{center}

Let $U\in Q_{n-1}(\Lambda_n)$, it is obvious that $U$ does not have $P_1$ as a direct summand since $P_1$ is sincere. We first consider that $U$ has $S_1$ as a direct summand. Since $\tau (1)=\substack{\\ \\2}\substack{1\\ \\ }\substack{\\ \\3}$, we know that $U$ does not have one of $\substack{\\ \\2}\substack{1\\ \\ }\substack{\\ \\3}$, $M_a$, the indecomposable module $N_2$ with $\mathsf{top}\ N_2=S_2$ and the indecomposable module $N_3$ with $\mathsf{top}\ N_3=S_3$ as a direct summand. Since $U\in Q_{n-1}(\Lambda_n)$, there exists exactly one idempotent $e_i$ with $i\neq 1$ satisfying $Ue_i=0$.
\begin{description}\setlength{\itemsep}{-3pt}
  \item[(i)] If $i=2$, the only possible direct summand $Y$ of $U$ satisfying $Ye_3\neq0$ is $\substack{1\\3}$ and the remaining direct summands give a $\tau$-tilting $\Lambda_{[4,n]}$-module $V$, so that $U=\substack{1}\oplus \substack{1\\3}\oplus V$. In this subcase, $\mathfrak{q}$ is defined by mapping $U$ to $\substack{1}\oplus \substack{1\\2}\oplus \substack{1\\3}\oplus V $. To observe that the latter one is included in $Q_n(\Lambda_n)\setminus  \mathcal{S}$, we have $\mathsf{Hom}_{\Lambda_n}(\substack{1\\2}\oplus U, \tau (\substack{1\\2}\oplus U))=0$ since $\tau(\substack{1\\2})=3$ and $\tau(\substack{1\\3})=2$. Moreover, it is easy to check that $\mathfrak{q}$ in this subcase is a bijection.

      If $i=3$, $U=\substack{1}\oplus \substack{1\\2}\oplus V$ with a $\tau$-tilting $\Lambda_{[4,n]}$-module $V$. Similarly, $\mathfrak{q}$ is defined by mapping $U$ to $\substack{1}\oplus \substack{1\\2}\oplus \substack{1\\3}\oplus V \in Q_n(\Lambda_n)\setminus  \mathcal{S}$, and $\mathfrak{q}$ in this subcase is also a bijection.

      Note that the number of $U$ in the cases of $i=2,3$ is $2a_{n-3}(\Lambda_{[4,n]})=2a_{n-3}({\bf{A}}_{n-3})$.

 \item[(ii)] If $i\geqslant 4$, the conditions $Ue_2\neq0$ and $Ue_3\neq 0$ must imply that $U=\substack{1}\oplus \substack{1\\2}\oplus \substack{1\\3}\oplus Z$ with $Ze_1=Ze_2=Ze_3=0$ and $Z\in Q_{n-4}(\Lambda_{[4,n]})$. We may regard $Z$ as a support $\tau$-tilting ${\bf{A}}_{n-3}$-module with support-rank $n-4$. By Lemma \ref{lemma-path-alg-type-A}, $Z$ is the left mutation of $\tau$-tilting $\Lambda_{[4,n]}$-module $V:=P\oplus Z$ with $P=[4,n]$. Then, $\substack{1}\oplus \substack{1\\2}\oplus \substack{1\\3}\oplus V \in Q_n(\Lambda_n)\setminus  \mathcal{S}$ is obvious. In this subcase, $\mathfrak{q}$ is defined by mapping $U$ to $\substack{1}\oplus \substack{1\\2}\oplus \substack{1\\3}\oplus V$ and it is a bijection. Besides, the number of $U$ in this subcase is $a_{n-4}(\Lambda_{[4,n]})=a_{n-4}({\bf{A}}_{n-3})$.
\end{description}
Similar to the situation in Proposition \ref{lambda-n-2}, the map $\mathfrak{q}$ defined in the case where $U$ has $S_1$ as a direct summand is no longer a bijection, but a surjection. If we count the number of modules in $Q_n(\Lambda_n)\setminus  \mathcal{S}$ in this case, there should be $a_{n-3}({\bf{A}}_{n-3})+a_{n-4}({\bf{A}}_{n-3})$ overlaps.

Next, we consider $U\in Q_{n-1}(\Lambda_n)$ such that $Ue_1\neq 0$ and $U$ does not have $S_1$ as a direct summand. There also exists exactly one idempotent $e_i$ with $i\neq 1$ satisfying $Ue_i=0$.
\begin{description}\setlength{\itemsep}{-3pt}
  \item[(iii)] If $i=2$, the only possible direct summand $Y$ of $U$ satisfying $Ye_1\neq0$ is $\substack{1\\3}$ so that $U=\substack{1\\3}\oplus V$ with $Ve_1=0$. Then, $P_1\oplus U \in Q_n(\Lambda_n)\setminus  \mathcal{S}$ is obvious and this is the only one possible case. In fact, by the analysis in (1) before, a completion of $\substack{1\\3}\oplus V$ to a $\tau$-tilting $\Lambda_n$-module which does not belong to $\mathcal{S}$ is either of form $P_1\oplus \substack{1\\3}\oplus W$ with $We_1=We_2=0$ or of form $\substack{1\\2}\oplus\substack{1\\3}\oplus\substack{1}\oplus W$ with $We_1=We_2=We_3=0$, but the latter one is excluded if we restrict  $Ve_1=0$. Therefore, $\mathfrak{q}$ in this subcase is defined by mapping $U$ to $P_1\oplus U$ and it is a bijection to the case (a2).

      Similarly, if $i=3$, we have $U=\substack{1\\2}\oplus V$ with $Ve_1=0$ and $P_1\oplus U \in Q_n(\Lambda_n)\setminus  \mathcal{S}$. Hence, $\mathfrak{q}$ is also defined by mapping $U$ to $P_1\oplus U$ and it is a bijection to the case (a1).

 \item[(iv)] Suppose $i\geqslant 4$. Then, $U$ has at least one of $\substack{1\\2}$, $\substack{1\\3}$, $\substack{\\ \\2}\substack{1\\ \\ }\substack{\\ \\3}$ and $M_a$ as a direct summand.
 \begin{itemize}\setlength{\itemsep}{-3pt}
 \item If $U$ has exactly one of $\substack{1\\2}$ and $\substack{1\\3}$ as a direct summand, say, $U=\substack{1\\2}\oplus V$ with $Ve_1=0$, $Ue_3\neq 0$ implies $S_3\subseteq \mathsf{top}\ V$. Then, $\tau(\substack{1\\2})=S_3$ indicates $\mathsf{Hom}_{\Lambda_n}(V, \tau(\substack{1\\2}))\neq 0$, contradicting with the assumption that $U$ is a $\tau$-rigid module. Also, one can get a contradiction for $U=\substack{1\\3}\oplus V$ with $Ve_1=0$.

   \item If $\substack{1\\2}\oplus \substack{1\\3}$ is a direct summand of $U$ such that $U=\substack{1\\2}\oplus\substack{1\\3}\oplus V$ with $Ve_1=0$, then $Ve_2=Ve_3=0$ by the similar analysis with Case (b) in (1). This implies that $V$ is a support $\tau$-tilting $\Lambda_{[4,n]}$-module with support-rank $n-4$. However, $\left | U \right |=2+n-4=n-2<n-1$, contradicting with $U\in Q_{n-1}(\Lambda_n)$.

   \item Otherwise, $U$ must have one of $\substack{\\ \\2}\substack{1\\ \\ }\substack{\\ \\3}$ and $M_a$ as a direct summand, so that $P_1\oplus U \in Q_n(\Lambda_n)\setminus  \mathcal{S}$ is well-defined by the analysis in (1). In this subcase, $\mathfrak{q}$ is defined by mapping $U$ to $P_1\oplus U$ and it is a bijection to the cases (a3), (a4) and (a5).
 \end{itemize}
\end{description}
We conclude that the map $\mathfrak{q}$ in this case is a bijection.

Lastly, it suffices to consider the following case.
\begin{description}\setlength{\itemsep}{-3pt}
 \item[(v)] If $Ue_1=0$, $P_1\oplus U \in Q_n(\Lambda_n)\setminus  \mathcal{S}$ is well-defined. In fact, it follows from  $\mathsf{Hom}_{\Lambda_n}(P_1, \tau U)=0$. Thus, $\mathfrak{q}$ in this case is also defined by mapping $U$ to $P_1\oplus U$ and it is a bijection. This corresponds to the Case (a) in (1).
\end{description}

Now, we have found that $\mathfrak{q}: Q_{n-1}(\Lambda_n)\rightarrow Q_n(\Lambda_n)\setminus  \mathcal{S}$ is a surjection, which is similar to the situation in Proposition \ref{lambda-n-2}. In particular, the reason why $\mathfrak{q}$ is not an injection is explained in cases (i) and (ii). Then, the number of modules in $Q_n(\Lambda_n)\setminus  \mathcal{S}$ is
\begin{center}
$a_{n-1}(\Lambda_n)-\#\left \{ 1\oplus \substack{1\\3}\oplus V \mid V \in Q_{n-3}(\Lambda_{[4,n]}) \right \}-\#\left \{ 1\oplus \substack{1\\2}\oplus\substack{1\\3}\oplus Z \mid Z\in Q_{n-4}(\Lambda_{[4,n]}) \right \}$.
\end{center}

On the other hand, the number of modules in $\mathcal{S}$ is $a_{n-3}(\Lambda_{[4,n]})=a_{n-3}({\bf{A}}_{n-3})$. Hence, we conclude that
\begin{center}
$\begin{aligned}
a_n(\Lambda_n)=&\#(Q_n(\Lambda_n)\setminus  \mathcal{S})+\#\mathcal{S}\\
=&a_{n-1}(\Lambda_n)-a_{n-3}({\bf{A}}_{n-3})-a_{n-4}({\bf{A}}_{n-3})+a_{n-3}({\bf{A}}_{n-3}).
\end{aligned}$
\end{center}
Note that $a_{n-4}({\bf{A}}_{n-3})=a_{n-3}({\bf{A}}_{n-3})$ by Lemma \ref{lemma-path-alg-type-A}. Then, the statement follows.
\end{proof}

\ \\

Department of Pure and Applied Mathematics, Graduate School of Information Science and Technology, Osaka University, 1-5 Yamadaoka, Suita, Osaka, 565-0871, Japan.

\emph{Email address}: \texttt{q.wang@ist.osaka-u.ac.jp}
\end{spacing}

\end{document}